\documentclass[a4paper,12pt]{article}

\usepackage{epsfig}
\usepackage{amsmath,amsthm}
\usepackage{amsfonts}
\usepackage{latexsym}
\usepackage{amsbsy}
\usepackage{amssymb,mathscinet}
\usepackage[all]{xy}

\numberwithin{equation}{section}

\usepackage{amsfonts,amssymb,amscd,amsmath,enumerate,verbatim,calc,eucal,enumerate}
\newcommand{\ma}{\mathcal}
\newcommand{\mf}{\mathfrak}
\newcommand{\m}{\CMcal}
\newcommand{\sK}{\mathcal{K}}
\theoremstyle{plain}
\newtheorem{theorem}{Theorem}[section]
\newtheorem{corollary}[theorem]{Corollary}

\newtheorem{proposition}[theorem]{Proposition}

\newtheorem{definition}[theorem]{Definition}

\begin{document}

\begin{center}
 \Large\bf{$\mathfrak{K}$-families and CPD-H-extendable families}
\end{center}

\vspace{1cm}

\begin{center} Santanu Dey and Harsh Trivedi
\end{center}

\vspace{1cm}

\vspace{0.5cm} {\bf Abstract:}  {We introduce,  for any set $S$, the concept of $\mf{K}$-family between two Hilbert $C^*$-modules over two $C^*$-algebras,
for a given completely positive definite (CPD-) kernel  $\mf{K}$ over $S$ between those $C^*$-algebras and  obtain a factorization theorem for such $\mf{K}$-families. If $\mf{K}$
is a CPD-kernel and $E$ is a full Hilbert $C^*$-module, then any
$\mf{K}$-family which is covariant with respect to a dynamical system $(G,\eta,E)$ on $E$, extends to a $\tilde{\mf{K}}$-family
on the crossed product $E \times_\eta G$, where $\tilde{\mf{K}}$ is a CPD-kernel.  Several characterizations
of $\mf{K}$-families, under the assumption that ${E}$ is full, are obtained
and covariant versions of these results are also given. One of these
characterizations says that such $\mf{K}$-families extend as CPD-kernels,
between associated (extended) linking algebras, whose $(2,2)$-corner is a
homomorphism and vice versa. We discuss a dilation theory of CPD-kernels
in relation to $\mf{K}$-families.}

\vspace{0.5cm}
\noindent {\bf keywords:} completely positive definite kernels, dilations, crossed product, Hilbert $C^*$-modules, Kolmogorov decomposition, linking algebras

\vspace{0.5cm}

\section{Introduction}

Let $\m B$ be a $C^*$-algebra and ${E}$ be a vector space which is a
right $\m B$-module satisfying $\alpha(xb)=(\alpha x)b=x(\alpha b)$
for $x\in {E},b\in \m B,\alpha\in\mathbb{C}$. The space ${E}$ is called an
{\it inner-product $\m B$-module} if there exists a mapping $\langle
\cdot,\cdot \rangle : E \times E \to \m{B}$ such that
\begin{itemize}
 \item [(i)] $\langle x,x \rangle \geq 0 ~\mbox{for}~ x \in {E}  $ and $\langle x,x \rangle = 0$ only if $x = 0 ,$
\item [(ii)] $\langle x,yb \rangle = \langle x,y \rangle b ~\mbox{for}~ x,y \in {E}$ and  $~\mbox{for}~ b\in \m B,  $
 \item [(iii)]$\langle x,y \rangle=\langle y,x \rangle ^*~\mbox{for}~ x ,y\in {E} ,$
\item [(iv)]$\langle x,\mu y+\nu z \rangle = \mu \langle x,y \rangle +\nu \langle x,z \rangle ~\mbox{for}~ x,y,z \in {E} $ and for $\mu,\nu \in \mathbb{C}$.
\end{itemize}
 An inner-product $\m B$-module ${E}$
which is complete with respect to the norm  $$\| x\| :=\|\langle
x,x\rangle\|^{1/2} ~\mbox{for}~ x \in {E}$$ 
is called a \textit{Hilbert $\m B$-module} or \textit{Hilbert $C^{*}$-module
over $\m B$}. It is said to be {\it full}
if the closure of the linear span of $\{\langle x,y\rangle:x,y\in{E}\}$ equals
$\m B$. We use the symbol [S] for the closure of the linear span of any set $S.$
Also for each $x\in{E}$ we use notation $|x|$ to denote $\langle
x,x\rangle^{1/2}$. Paschke and Rieffel (cf. \cite{Ri74}, \cite{Pas73}) contributed to the theory of
Hilbert $C^*$-modules immensely in early 1970s and it finds applications in the classification of $C^*$-algebras,
the dilation theory of semigroups of completely positive maps, the theory of quantum groups, etc..

Apart from the notion of Hilbert $C^*$-module, the property of complete positivity is a key concept needed in this article.
  A linear mapping $\tau$ from a $C^*$-algebra $\m B$ to a $C^*$-algebra $\m C$ is called {\it completely positive} if
for each $n\in \mathbb N$, $\sum_{i,j=1}^n c_j^{*}
\tau(b_j^{*}b_i)c_i\geq 0$ where $b_1,b_2,\ldots,b_n$ are from $\m
B$ and $c_1,c_2,\ldots,c_n$ are from $\m C$. The theory of completely positive
maps plays an important role in operator algebras, quantum statistical mechanics,  quantum information theory, etc.. Completely positive maps
between unital $C^*$-algebras are characterized by the Paschke's
GNS construction (cf. Theorem 5.2, \cite{Pas73}).
Let ${E}$ be a Hilbert $\m B$-module, ${F}$ be a Hilbert $\m
C$-module and $\tau$ be a linear map from $\m B$ to $\m C$. A map
${T}:{E}\to{F}$ is called {\it $\tau$-map} if
$$\langle {T}(x),{T}(y)\rangle=\tau(\langle x,y\rangle)~\mbox{for all} ~x,y\in{E}.$$ Skeide in
\cite{Sk12} developed a factorization theorem for $\tau$-maps when
$\tau$ is completely positive based on Paschke's GNS contruction. This theorem
generalizes the Stinespring type theorem for Hilbert
$C^*$-modules due to Bhat, Ramesh and Sumesh (cf. \cite{BRS12}). Certain related covariant versions 
of this theorem have been explored in \cite{J11} and \cite{He99}.

The following definition of completely positive definite (CPD-) kernels on arbitrary set $S$, which plays a crucial role in exploring the 
theory of CPD-semigroups over $S$, is from \cite{BBLS04}:
\begin{definition}
Let $\m B$ and $\m C$ be $C^*$-algebras. By $\ma
B(\m B,\m C)$ we denote the set of all bounded linear maps from  $\m B$ to $\m C$. For a set $S$
we say that a mapping $\mf{K}:S\times S \to \ma B(\m B,\m C)$ is a {\rm completely positive definite kernel} or a {\rm CPD-kernel} over $S$ from $\m B$ to $\m C$ if
\[
 \sum_{i,j} c^*_i \mf{K}^{\sigma_i, \sigma_j} (b^*_i b_j) c_j \geq 0~\mbox{for all finite choices of $\sigma_i\in S$, $b_i\in \m B$, $c_i\in \m C$. }~
\]
\end{definition}
The notion of a completely multi-positive map  which is introduced in \cite{He99} is an example of a CPD-kernel  over the finite set $S = \{1,\ldots, n\}$. 
 CPD-kernels over the set $S = \{0, 1\}$ and semigroups of CPD-kernels were first studied by Accardi and Kozyrev in \cite{AK01}. 
Motivated by the definition of $\tau$-map, we define $\mf{K}$-family, where $\mf{K}$ is a CPD-kernel, in Section \ref{sec1}. 
Some of the results about $\tau$-maps from \cite{Sk12} and \cite{SSu14} are extended to $\mf{K}$-families in this article.

In Section \ref{sec1}, for a CPD-kernel $\mf{K}$ we  show that
any $\mf{K}$-family $\{\ma K^{\sigma}\}_{\sigma \in S}$ factorizes in terms of a $C^*$-correspondence $\m F$, a mapping from the set $S$ to $\m F$ and an isometry, if the corresponding $C^*$-algebras
are assumed to be unital. The factorization result is a Stinespring type theorem. Further, we prove a covariant version of this
theorem in terms of the following notions: Let $G$ be a locally compact group and let $\m B$ be a $C^*$-algebra. We call a group homomorphism $\alpha:G\to Aut(\m B)$
an {\it action of $G$ on $\m B$}. If $t\mapsto \alpha_{t}(b)$ is
continuous for all $b\in\m B$, then we call $(G,\alpha,\m B)$ a {\it
$C^*$-dynamical system}. We denote by $\m U \m B$  the group
of all unitary elements of the $C^*$-algebra $\m B$ and use symbol
$\alpha_t$ for the image of $t\in G$ under $\alpha$.

\begin{definition}
Let $S$ be a set and let $ \mf{K}:S\times S\to \ma B(\m B,\m C)$ be
a kernel over $S$ with values in the bounded maps from a
$C^*$-algebra $\m B$ to a unital $C^*$-algebra $\m C$.  Let $u:G\to
\m U\m C$ be a unitary representation of a locally compact group $G$. The kernel $\mf{K}$ is called {\rm $u$-covariant} with respect to the  $(G,\alpha,\m B)$ if for all
$\sigma,\sigma'\in S$
\[
 \mf{K}^{\sigma,\sigma'}(\alpha_t(b))=u_t \mf{K}^{\sigma,\sigma'}(b) u^*_t~\mbox{for}~b\in\m B,~t\in G.
\]
\end{definition}

Let ${E}$ and ${F}$ be Hilbert $C^*$-modules over a $C^*$-algebra
$\m B$. A map $T:{E}\to{F}$ is called {\it adjointable} if there
exists a map $T':{F}\to{E}$ such that
 \[
 \langle T (x),y\rangle =\langle x,T'(y) \rangle~\mbox{for all}~x\in {E}, y\in {F}.
 \]
The map $T'$ is unique for each $T$ and we denote it by $T^{*}$. We denote
the set of all adjointable maps from ${E}$ to ${F}$ by $\ma B^a
({E},F)$, and if ${E}={F}$, then we denote by $\ma B^a ({E})$ the space $\ma B^a
({E},{E})$. The set of all bounded right linear maps from
$E$ into $F$ will be denoted by $\ma B^r (E,F)$.
Let  ${E}$ be a Hilbert $\m
B$-module and let ${F}$ be a Hilbert $\m C$-module. A map
$\Psi:E\to{F}$  is said to be a {\it morphism of Hilbert
$C^*$-modules} if there exists a $C^*$-algebra homomorphism $\psi:\m
B\to\m C$ such that
$$\langle \Psi(x),\Psi(y)\rangle=\psi(\langle
x,y\rangle)~\mbox{for all}~ x,y\in{E}.$$ 
If ${E}$ is full, then
$\psi$ is unique for $\Psi$. A bijective map $\Psi:{E}\to{F}$ is
called an {\it isomorphism of Hilbert $C^*$-modules} if $\Psi$ and
$\Psi^{-1}$ are morphisms of Hilbert $C^*$-modules. We denote the
group of all isomorphisms of Hilbert $C^*$-modules from ${E}$ to
itself by $Aut({E})$.

\begin{definition}\label{def8}
Let $G$ be a locally compact group and let $\m B$ be a
$C^*$-algebra. Let $E$ be a full Hilbert $\m B$-module. A group
homomorphism $t\mapsto\eta_{t}$ from $G$ to $Aut({E})$ is called a
{\rm continuous  action of $G$ on ${E}$} if $t\mapsto \eta_{t}(x)$
from $G$ to ${E}$ is continuous for each $x\in E$. In this case we
call the triple $(G,\eta,{E})$ a {\rm dynamical system on the
Hilbert $\m B$-module ${E}$}.

\end{definition}

Any $C^*$-dynamical system $(G,\alpha,\m B)$ can be regarded as a dynamical system on the
Hilbert $\m B$-module $\m B$.  In Section \ref{sec1} we also examine the extendability of covariant $\mf{K}$-families with respect to any dynamical system $(G,\eta,{E})$ on a
Hilbert $C^*$-module $E$ to the crossed product Hilbert $C^*$-module $E \times_\eta G$.
Let $E^*:=\{x^*: x\in E\}\subset \ma B^a ({E}, \m B )$ where
$x^* y:=\langle x,y\rangle$ for all $x,y\in E$.  Then $\ma K(E):=[EE^*]$ is a
$C^*$-subalgebra of $\ma B^a ({E})$. Indeed, $E^*$ is a Hilbert $\ma
K(E)$-module where $\langle x^*,y^*\rangle:= xy^*$ for all $x,y\in
E$.  The (extended) {\it linking algebra} of $E$ is defined by
$$\m
L_{{E}}:=
\begin{pmatrix}
{\m B} & {{{E}}^*}  \\
{{E}} & \ma B^a ({E})
\end{pmatrix} \subset \ma B^a (\m B\oplus {E} ) .$$
(cf. \cite{Sk01}). It is shown in Section \ref{sec2} that for any CPD-kernel $\mf{K}$, the  $\mf{K}$-family on full Hilbert
$C^*$-modules is same as the set of maps defined on the  Hilbert $C^*$-modules which extend as a CPD-kernel between their linking algebras. A characterization of 
such  $\mf{K}$-families is obtained in terms of completely bounded  maps between certain Hilbert $C^*$-modules. We derive the covariant 
versions of the above results too.  
In Section \ref{sec3}, as an application of our theory we propose and explore a new dilation theory of any CPD-kernel $\mf{K}$ associated to a family of maps between certain Hilbert $C^*$-modules. This dilation is called a CPDH-dilation and under additional assumptions, the family of maps between the Hilbert $C^*$-modules becomes a $\mf{K}$-family.

\section{$\mf{K}$-families and crossed products of Hilbert \\ $C^*$-modules}\label{sec1}

\begin{definition}
   Let ${E}$ and ${F}$ be Hilbert $C^*$-modules over $C^*$-algebras $\m B$ and $\m C$ respectively. Let $S$ be a set and let $ \mf{K}:S\times S\to \ma B(\m B,\m C)$ be a kernel. 
Let $\ma K^{\sigma}$ be a map from $E$ to $F$ for each $\sigma\in S$. The family $\{\ma K^{\sigma}\}_{\sigma\in S}$ is called {\rm $\mf{K}$-family} if
   \[
    \langle \ma K^{\sigma} (x),\ma K^{{\sigma}'}(x')\rangle=\mf{K}^{\sigma,\sigma'}(\langle x, x'\rangle),~\mbox{for}~x,x' \in E,~\sigma,\sigma'\in S.
   \]

\end{definition}

Let $\m A$ and $\m B$ be  $C^*$-algebras. We recall that a {\it $C^*$-correspondence} from $\m
A$ to $\m B$ is defined as a right Hilbert $\m B$-module $E$ together with a
$*$-homomorphism $\phi:\m A\to \ma B^a(E)$ where $\ma B^a(E)$ is the
set of all adjointable operators on $E$. The left action of $\m A$
on $E$ given by $\phi$ is defined as
\[
ay:=\phi(a)y~\mbox{for all}~a\in \m A,y\in E.
\]
The following theorem  deals with the factorization of 
$\mf{K}$-families:
\begin{theorem}\label{thm1}
 Let $\m B$ and $\m C$ be $C^*$-algebras where $\m B$ is unital. Let ${E}$ and ${F}$ be Hilbert $C^*$-modules over $\m B$ and $\m C$ respectively, 
and $S$ be a set. If $\ma K^{\sigma}$ is a map from $E$ to $F$ for each $\sigma\in S$, 
then the following conditions are equivalent:
  \begin{itemize}
 \item [(i)] $\{\ma K^{\sigma}\}_{\sigma\in S}$ is a $\mf{K}$-family where $\mf{K}:S\times S\to \ma B(\m B,\m C)$ is a CPD-kernel.
 \item [(ii)] There exists a pair $(\m F, \mf{i})$ consisting of a $C^*$-correspondence $\m F$ from $\m B$ to $\m C$ and 
a map $\mf{i}: S\to \m F$, and there exists an isometry $
\nu:{E}\bigodot\m F\to{F}$ such that
\begin{align}\label{eqn*}
  \nu(x\odot \mf{i}(\sigma))=\ma{K^{\sigma}}(x)~\mbox{for all}~x\in{E},~\sigma\in S.
 \end{align}\end{itemize}
\end{theorem}
\begin{proof}
Suppose (ii) is given. For each $\sigma, \sigma'\in S$ we define
$\mf{K}^{\sigma,\sigma'}:\m B\to\m C$ by
$\mf{K}^{\sigma,\sigma'}(b):= \langle \mf{i}(\sigma), b
\mf{i}(\sigma')\rangle$ for $b\in \m B$. The mapping $\mf{K}$ is a CPD-kernel,
for
\begin{align*}
  \sum_{i,j} c^*_i \mf{K}^{\sigma_i, \sigma_j} (b^*_i b_j) c_j
 =\sum_{i,j} c^*_i \langle \mf{i}(\sigma_i), b^*_i b_j
\mf{i}(\sigma_j)\rangle  c_j  = \left<\sum_{i}b_i\mf{i}(\sigma_i)c_i
, \sum_{j} b_j \mf{i}(\sigma_j)c_j\right>
 \geq 0
\end{align*}
for all finite choices of $\sigma_i\in S$, $b_i\in \m B$, $c_i\in \m
C$. Further for $x,x' \in E,~\sigma,\sigma'\in S$ we have

   \[
    \langle \ma K^{\sigma} (x),\ma K^{{\sigma}'}(x')\rangle=
     \langle \nu(x\odot\mf{i}(\sigma)) ,\nu(x'\odot\mf{i}(\sigma'))\rangle=\mf{K}^{\sigma,\sigma'}(\langle x, x'\rangle).
   \]
So $\{\ma K^{\sigma}\}_{\sigma\in S}$ is a $\mf{K}$-family, i.e., (i) holds.

Conversely, suppose (i) is given. By Kolmogorov decomposition for $\mf{K}$ (cf.
Theorem 3.2.3 of \cite{BBLS04} and Theorem 4.2 of \cite{Sk11}) we get a
pair $(\m F, \mf{i})$ consisting of a $C^*$-correspondence $\m F$
from $\m B$ to $\m C$ and a map $\mf{i}: S\to \m F$ such that
$\m F=[\{b\mf{i}(\sigma)c:b\in\m B,c\in\m C,\sigma\in S\}]$. We use
symbol $E\bigodot\m F$ for the interior tensor product of ${E}$ and
$\m F$.  Define a linear map $\nu:{E} \bigodot\m F\to{F}$ by
$\nu(x\odot b\mf{i}(\sigma) c):={\ma K^\sigma}(xb)c~\mbox{for
all}~x\in{E},~b\in\m B,~c\in\m C,~\sigma\in S$. We have
\begin{eqnarray*}
 \langle \nu(x\odot b\mf{i}(\sigma) c), \nu(x' \odot b'\mf{i}(\sigma') c')\rangle &=& \langle\ma K^\sigma (xb)c,\ma K^{\sigma'}(x' b')c'\rangle=c^*\mf{K}^{\sigma,\sigma'}(\langle xb,x' b'\rangle)c'\\ &=&\langle \mf{i}(\sigma) c,(\langle xb,x' b'\rangle) \mf{i}(\sigma') c'\rangle =\langle x\odot b\mf{i}(\sigma) c,x'\odot b'\mf{i}(\sigma') c'\rangle
\end{eqnarray*}
for all $x,x'\in{E}$, $b,b'\in\m B$, $c,c'\in\m C$ and
$\sigma,\sigma'\in S$. Hence $\nu$ is an isometry satisfying
equation \ref{eqn*}. This proves ``(i) $\Rightarrow$ (ii)''
\end{proof}

We now examine the covariant version of the above theorem.
If $(G,\eta,{E})$ is a dynamical system on a full Hilbert $\m B$-module ${E}$, then there exists unique $C^*$-dynamical system
$(G,\alpha^{\eta},\m B)$ (cf. p.806 of \cite{J11}) such that
$$\alpha^{\eta}_t(\langle x,y \rangle)=\langle {\eta}_{t}(x),
{\eta}_{t}(y)\rangle~\mbox{for all}~x,y\in {E}~\mbox{and}~t\in G.$$
Moreover, for all $x\in {E}$ and $b\in \m B$ we
have $\eta_t (xb)=\eta_t(x) \alpha^{\eta}_t (b)$.

\begin{definition}
Let $\m C$, $\m D$ be unital $C^*$-algebras, and let $u:G\to \m U\m
C$ and $u':G\to \m U\m D$ be unitary representations on a locally
compact group $G$. Let ${E}$ be a full Hilbert $C^*$-module over a
$C^*$-algebra $\m B$ and let ${F}$ be a $C^*$-correspondence from
$\m D$ to $\m C$. Let $S$ be a set and  $(G,\eta,{E})$ be a
dynamical system on $E$. Consider bounded linear maps $\ma K^\sigma
:{E}\to{F}$ for $\sigma\in S$. Then the family $\{\ma
K^\sigma\}_{\sigma\in S}$ is called {\rm $(u',u)$-covariant with
respect to the dynamical system $(G,\eta,{E})$} if
$$\ma K^\sigma({\eta_t}({x}))=u'_t \ma K^\sigma(x)u^*_t~\mbox{for each $t\in G$, $\sigma\in S$ and ${x}\in{E}$.}$$

\end{definition}

\begin{theorem}\label{thm2}
  Let $u:G\to \m U\m C$, $u':G\to \m U\m D$ be unitary representations of a locally compact group $G$ on unital $C^*$-algebras $\m C$ and $\m D$ respectively. 
Let ${E}$ be a full Hilbert $C^*$-module over a unital $C^*$-algebra $\m B$,  ${F}$ be a
$C^*$-correspondence from $\m D$ to $\m C$ and $S$ be a set. Let $\ma
K^{\sigma}$ be a map from $E$ to $F$ for each $\sigma\in S$. If
$(G,\eta,{E})$ is a dynamical system on ${E}$, then the following
conditions are equivalent:
 \begin{itemize}
\item [(i)] $\{\ma K^{\sigma}\}_{\sigma\in S}$ is a $(u',u)$-covariant $\mf{K}$-family with respect to the
dynamical system $(G,\eta,{E})$ where $\mf{K}:S\times S\to \ma B(\m B,\m C)$ is a CPD-kernel.
 \item [(ii)] There exists a pair $(\m F, \mf{i})$ consisting of a $C^*$-correspondence $\m F$ from $\m B$ to $\m C$ and a map $\mf{i}: S\to \m F$, an isometry $
\nu:{E}\bigodot\m F\to{F}$ such that
\[
  \nu(x\odot \mf{i}(\sigma))=\ma{K^{\sigma}}(x)~\mbox{for all}~x\in{E}, \sigma \in S,
 \]
and  unitary representations $v:G\to \m U \ma B^a (\m F)$
and $w':G\to \m U \ma B^a ({E}\bigodot \m F)$ such that
\begin{enumerate}
\item [(a)] $\pi(\alpha^{\eta}_{t}(b))=v_{t}\pi(b)v^{*}_{t}$ for all $b\in\m B,t\in G,$
\item [(b)] $v_{t}\mf{i}(\sigma) =\mf{i}(\sigma)  u_{t}$ for all $t\in G$ and $\sigma\in S$,
\item [(c)]$w'_t(x\odot b\mf{i}(\sigma) c):=\eta_t(x)\odot v_t (b\mf{i}(\sigma) c)$ for all $b\in\m B$, $c\in \m C$, $x\in{E}$, $\sigma\in S$ and $t\in G$,

\item [(d)] $\nu w'_t=u'_t\nu$ for all $t\in G.$
\end{enumerate}
\end{itemize}
\end{theorem}

\begin{proof} Suppose statement $(ii)$ is given. The collection $\{\ma K^{\sigma}\}_{\sigma\in S}$ is a $\mf{K}$-family where
$\mf{K}^{\sigma,\sigma'}:\m B\to\m C$ is defined by
$\mf{K}^{\sigma,\sigma'}(b):= \langle \mf{i}(\sigma), b
\mf{i}(\sigma')\rangle$ for $b\in \m B$ and $\sigma, \sigma'\in S$.
Also
\begin{eqnarray*}
\ma
K^{\sigma}(\eta_t(x))&=&\nu(\eta_t(x)\odot\mf{i}(\sigma))=\nu(\eta_t(x)\odot
v_t v_{t^{-1}}\mf{i}(\sigma))=\nu w'_t(x\odot
v_{t^{-1}}\mf{i}(\sigma))\\ &=& u'_t\nu (x\odot
v_{t^{-1}}\mf{i}(\sigma))=u'_t\nu (x\odot\mf{i}(\sigma)
u_{t^{-1}})=u'_t\nu (x\odot\mf{i}(\sigma)) u_{t^{-1}}=u'_t \ma
K^{\sigma}(x) u_{t^{-1}}
\end{eqnarray*}
for all $x\in {E}$, $\sigma\in S$ and $t\in G$. Hence statement
$(i)$ holds. Conversely, let us assume that (i) holds. The kernel $\mf{K}$ is $u$-covariant because for each $\sigma,\sigma'\in S$
\begin{eqnarray*}
 \mf{K}^{\sigma,\sigma'}(\alpha^\eta_t(\langle x,x'\rangle))&=&\mf{K}^{\sigma,\sigma'} (\langle \eta_t(x),\eta_t(x')\rangle)=\langle\ma K^{\sigma}(\eta_t (x)),\ma K^{\sigma'}(\eta_t (x'))\rangle\\ &=& \langle u'_t \ma K^{\sigma} (x)u^*_t,u'_t \ma K^{\sigma'} (x')u^*_t\rangle = u_t \langle\ma K^{\sigma} (x), \ma K^{\sigma'} (x')\rangle u^*_t \\ &=& u_t \mf{K}^{\sigma,\sigma'}(\langle x,x'\rangle)u^*_t~\mbox{for all}~x,x'\in{E},t\in G.
\end{eqnarray*}

By Theorem \ref{thm1} or Kolmogorov decomposition we get a pair $(\m F, \mf{i})$
consisting of a $C^*$-correspondence $\m F$ from $\m B$ to $\m C$
where the left action is given by a $*$-homomorphism $\pi:\m B\to
\ma B^a (\m F)$ and a map $\mf{i}: S\to \m F$ such that
$[\{b\mf{i}(\sigma)c:b\in\m B,c\in\m C,\sigma\in S\}]=\m F$. Further
we have an isometry $\nu:{E} \bigodot\m F\to{F}$ defined by
 \[
  \nu(x\odot b\mf{i}(\sigma) c):={\ma K^\sigma}(xb)c~\mbox{for all}~x\in{E},~b\in\m B,~c\in\m C,~\sigma\in S.
 \]
For each $t\in G$, set $v_t(b\mf{i}(\sigma)
c):=\alpha^\eta_t(b)\mf{i}(\sigma) u_{t}c $ for all $t\in G$,
$b\in\m B$, $c\in \m C$ and $\sigma\in S$. Observe that
\begin{eqnarray*}
\langle v_{t}(b\mf{i}(\sigma) c), v_{t}(b'\mf{i}(\sigma') c')\rangle
&=&\langle \alpha^\eta_t(b)\mf{i}(\sigma) u_{t}c
,\alpha^\eta_t(b')\mf{i}(\sigma') u_{t}c' \rangle=(u_t
c)^{*}\mf{K}^{\sigma,\sigma'}(\alpha^\eta_t(b)^{*}\alpha^\eta_t(b'))u_t
c'\\ &=& c^{*} u_t^{*}u_t\mf{K}^{\sigma,\sigma'}(b^{*} b')u_t^{*}u_t
c'=\langle  b\mf{i}(\sigma) c, b'\mf{i}(\sigma') c'\rangle
\end{eqnarray*}
for all $b,b'\in\m B$, $\sigma,\sigma'\in S$ and $c,c'\in\m C$. Since $\alpha^{\eta}_{t}$ is an automorphism
and $u_{t}$ is a unitary for each $t \in G$, it is immediate that $v_t$ extends
uniquely to a unitary $v_t:\m F\to\m F$ for each $t \in G$. Because of the 
continuity of $t\mapsto \alpha^{\eta}_{t}(b)$ for each $b\in\m B$,
the continuity of $u$ and the fact that  $v_t$ is  a unitary for each $t\in G$, it follows that $t \mapsto v_tf$
is  continuous for each $f \in \m F$. Hence $v:G\to \m U \ma B^a (\m F)$ is a unitary
representation. $\mbox{For all}~b,b'\in\m B,~t\in G,~c\in \m C$ we
get
\begin{eqnarray*}
\pi(\alpha^\eta_{t}(b'))(b\mf{i}(\sigma)
c)&=&(\alpha^\eta_{t}(b')b)\mf{i}(\sigma) c=
v_t(b'\alpha^\eta_{t^{-1}}(b)\mf{i}(\sigma) u_{t^{-1}}c)\\ &=&
v_t\pi(b')(\alpha^\eta_{t^{-1}}(b)\mf{i}(\sigma) u_{t^{-1}}c)=
v_{t}\pi(b')v_{t^{-1}}(b\mf{i}(\sigma) c).
\end{eqnarray*} Thus $v$ satisfies conditions (a) and (b).
For each $t\in G$, define $w'_t:{E}\bigodot \m F\to E\bigodot \m F$
by
\[
w'_t(x\odot b\mf{i}(\sigma) c):=\eta_t(x)\odot v_t b\mf{i}(\sigma) c
\]
for all $b\in\m B$, $c\in \m C$, $\sigma\in S$, $x\in{E}$. We get
\begin{eqnarray*}
&&\langle w'_t( x \odot b\mf{i}(\sigma) c), w'_t(x'\odot
b'\mf{i}(\sigma') c')\rangle
 =\langle v_t(
b\mf{i}(\sigma) c) , \langle {\eta}_{t}(x), {\eta}_{t}(x')\rangle
v_t (b'\mf{i}(\sigma') c') \rangle\\ &=&\langle v_t (
b\mf{i}(\sigma) c ), \alpha^{\eta}_t(\langle x,x' \rangle) v_t
(b'\mf{i}(\sigma') c') \rangle = \langle v_t (b\mf{i}(\sigma) c) ,
v_t(\langle x,x' \rangle) b'\mf{i}(\sigma') c')\rangle\\& =& \langle
b\mf{i}(\sigma) c, \langle x,x'\rangle b'\mf{i}(\sigma')
c'\rangle=\langle x\odot b\mf{i}(\sigma) c, x'\odot
b'\mf{i}(\sigma') c'\rangle
\end{eqnarray*}
for all $b,b'\in\m B$, $c,c'\in \m C$, $x,x'\in E$,
$\sigma,\sigma\in S$. Using the strict continuity of $v$ and the
continuity of $t\mapsto\eta_{t}(x)$ for all $x\in{E}$ we obtain that
the map $t\mapsto w'_t z$ is continuous on finite sums of elementary
tensors $z\in {E}\bigodot \m F$. Now $\|w'_t\|\leq 1$ implies $w'$
is strictly continuous and therefore a unitary representation.
Moreover, we have
\begin{eqnarray*}
\nu w'_t(x\odot b\mf{i}(\sigma) c)&=&\nu(\eta_t(x)\odot v_t
(b\mf{i}(\sigma)
c ))=\nu(\eta_t (x)\odot \alpha^{\eta}_t (b)\mf{i}(\sigma) u_t c)\\
&=&\ma K^{\sigma}(\eta_t(x)\alpha^{\eta}_t(b))u_t c=\ma
K^{\sigma}(\eta_t(xb))u_t c= u'_t\ma K^{\sigma}(xb)u^{*}_tu_t
c\\
&=& u'_t\ma K^{\sigma}(xb)c= u'_t\nu(x\odot b\mf{i}({\sigma}) c)
\end{eqnarray*}
for all $b\in\m B$, $c\in \m C$, $x\in{E}$, $\sigma\in S$ and $t\in
G$.
\end{proof}

The next corollary proves the uniqueness of the above theorem.

\begin{corollary}

Let $\m E$ be another $C^*$-correspondence from $\m D$ to $\m C$. For
$\sigma\in S$, let $\tilde{\mu}^{\sigma}:{E}\to\m E$ be maps such
that $[\{\tilde{\mu}^\sigma({E})\m C:\sigma\in S\}]=\m E$ and let
$\tilde{\nu}:\m E\to{F}$ be an isometry such that
$\tilde{\nu}\tilde{\mu}^{\sigma}=\ma K^{\sigma}$. Then there exists
a unitary representation $w^{\prime\prime}_t:G\to \m U\ma B^a (\m
E)$ defined by
\[
 w^{\prime\prime}_t(\tilde{\mu}^{\sigma}(x)c)=\tilde{\mu}^{\sigma}(\eta_{t}(x))u_t c~\mbox{for $x\in{E}$, $t\in G$, $\sigma\in S$ and $c\in \m C$}
\]
and a unitary $u:\m E\to{E}\bigodot\m F$ defined by
$u:\tilde{\mu}^{\sigma}(x)\mapsto x\odot\mf{i}(\sigma)$, where
$\sigma\in S$ and $(\m F,\mf{i})$ is the Kolmogorov decomposition
for kernel $\mf{K}$  such that
\begin{enumerate}
\item [(a)] $\nu u=\tilde{\nu}$, $uw^{\prime\prime}_t=w^{\prime}_t u$ for all $t\in G$ and
\item [(b)] $u\tilde{\mu}^{\sigma}=\mu^{\sigma}$, where for $\sigma\in S$ the mapping $\mu^\sigma:{E}\to{E}\bigodot\m F$ is defined by $x\mapsto x\odot\mf{i}(\sigma).$
\end{enumerate}
\end{corollary}
\begin{proof}
For all $x,x'\in {E}$, $c,c'\in \m C$, $\sigma,\sigma'\in S$ we have
\begin{eqnarray*}
&&\langle \tilde{\mu}^\sigma(\eta_t(x))u_t c,
\tilde{\mu}^{\sigma'}(\eta_t(x'))u_t c'\rangle =  \langle \ma
K^\sigma(\eta_t(x))u_t c, \ma K^{\sigma'}(\eta_t(x'))u_t c'\rangle
\\ &=& \langle u_t c,\mf{K}^{\sigma,\sigma'}(\alpha_t(\langle x,x'\rangle))
u_t c'\rangle =\langle \ma K^\sigma (x) c,\ma
K^{\sigma'}(x')c'\rangle =\langle \tilde{\mu}^\sigma(x)c,
\tilde{\mu}^{\sigma'}(x')c'\rangle.
\end{eqnarray*}
Therefore $w^{\prime\prime}$ is a unitary representation.
\end{proof}

Let $\m B$ be a $C^*$-algebra and let $G$ be a locally compact
group. Let $(G,\eta,{E})$ be a dynamical system on a full Hilbert
$\m B$-module ${E}$. The crossed product $E\times_{\eta}G$ (cf.
\cite{Kas88},\cite{EKQR00}) is the completion of an inner-product
$\m B\times_{\alpha^{\eta}}G$-module $C_c(G,{E})$ where the module
action and the $\m B\times_{\alpha^{\eta}}G$-valued inner product
are given by
\begin{align*}\label{eqn1} lg(s)&=\int_{G}
l(t)\alpha^{\eta}_{t}(g(t^{-1}s))dt,\\
\langle l,m\rangle_{\m B\times_{\alpha^{\eta}}G }(s)&=\int_{G}
\alpha^{\eta}_{t^{-1}}(\langle l(t),m(ts)\rangle)dt
\end{align*}
respectively, for $g\in C_c (G,\m B)$ and $l,m\in C_c (G,{E})$.
We derive for any CPD-kernel $\mf{K}$ the extendability of a covariant $\mf{K}$-family to that on the crossed
product of the Hilbert $C^*$-module corresponding to the given
dynamical system.

\begin{proposition}\label{prop2}
Let $S$ be a set and let $\mf{K}:S\times S\to \ma B(\m B,\m C)$ be a
CPD-kernel over $S$ from a unital $C^*$-algebra $\m B$ to a unital
$C^*$-algebra $\m C$. Let $\m D$ be a unital $C^*$-algebra, and let
$u:G\to \m U\m C$ and $u':G\to \m U\m D$ be unitary representations of
a locally compact group $G$. Suppose ${E}$ is a full Hilbert $\m B$-module,
 ${F}$ is a $C^*$-correspondence from $\m D$ to $\m C$ and
$\ma K^{\sigma}$ is a map from $E$ to $F$ for each $\sigma\in S$.  If $\{\ma
K^{\sigma}\}_{\sigma\in S}$ is a $(u',u)$-covariant $\mf{K}$-family
with respect to the dynamical system $(G,\eta,{E})$, then there
exists a family of maps $\tilde{\ma K}^{\sigma}:{E}\times_{\eta}
G\to {F}$ such that
\[
\tilde{\ma K}^{\sigma}(l)=\int_{G} {\ma
K}^{\sigma}(l(t))u_{t}dt~\mbox{for all}~ l\in C_c (G,{E}),~\sigma\in
S
\]
and there exists a CPD-kernel
$\tilde{\mf{K}}^{\sigma,\sigma'}:\m B\times_{\alpha^{\eta}}G\to \m
C$, which satisfies
\[
\tilde{\mf{K}}^{\sigma,\sigma'}(f)=\int_{G}
{\mf{K}^{\sigma,\sigma'}}(f(t))u_{t}dt~\mbox{for all}~ f\in C_c
(G,\m B),~\sigma,\sigma'\in S,
\]
such that $\{ \tilde{\ma
K}^{\sigma}\}_{\sigma\in S}$ is a $\tilde{\mf{K}}$-family.
\end{proposition}
\begin{proof}
Let $(\m F,\mf{i})$ be the covariant Kolmogorov decomposition
associated with the CPD-kernel  $\mf{K}:S\times S\to \ma B(\m B,\m
C)$ described in Theorem \ref{thm2}. Consider maps
$\tilde{\mf{K}}^{\sigma,\sigma'}:\m B\times_{\alpha^{\eta}}G\to \m C$
defined by
\[
\tilde{\mf{K}}^{\sigma,\sigma'}(f):=\langle\mf{i}(\sigma),(\pi\times
v)(f)\mf{i}(\sigma')\rangle~\mbox{for all}~ f\in C_c (G,\m
B),~\sigma,\sigma'\in S.
\]
Similar computations as in Theorem \ref{thm1} proves that
$\tilde{\mf{K}}$ is a CPD-kernel on $S$ from $\m B\times_{\alpha^{\eta}}G$
to $\m C$. For $~\sigma,\sigma'\in S$
\begin{eqnarray*}
\tilde{\mf{K}}^{\sigma,\sigma'}(f)&=&\langle\mf{i}(\sigma),(\pi\times
v)(f)\mf{i}(\sigma')\rangle=\langle\mf{i}(\sigma),\int_{G}
\pi(f(t))v_t \mf{i}(\sigma') dt\rangle\\ &=& \int_{G}
\langle\mf{i}(\sigma), \pi(f(t))v_t \mf{i}(\sigma') \rangle
dt=\int_{G} \langle\mf{i}(\sigma), \pi(f(t)) \mf{i}(\sigma') u_t
\rangle dt
\\&=&\int_{G} \langle\mf{i}(\sigma), \pi(f(t)) \mf{i}(\sigma') \rangle u_t dt=\int_{G} {\mf{K}}^{\sigma,\sigma'}(f(t))u_{t}dt~\mbox{for all}~ f\in C_c
(G,\m B).
\end{eqnarray*}
The third equality in the above equation array, follows by applying Lemma 1.91 of
\cite{Wi07} for a bounded linear map $L:\ma B^a (\m F)\to \m C$
which is defined as $L
(T):=\langle\mf{i}(\sigma),T\mf{i}(\sigma')\rangle$ for
all $T\in \ma B^a (\m F)$. Define $\tilde{\ma
K}^{\sigma}:{E}\times_{\eta} G\to {F}$ by
\[
\tilde{\ma K}^{\sigma}(l):=\int_{G} {\ma K^{\sigma}}(l(t))
u_{t}dt~\mbox{for all}~ \sigma\in S,~l\in C_c (G,{E}).
\]
From Theorem \ref{thm2} we get an isometry $ \nu:{E}\bigodot\m
F\to{F}$ such that
\[
  \nu(x\odot \mf{i}(\sigma))=\ma{K^{\sigma}}(x)~\mbox{for all}~x\in{E},~\sigma\in S,
 \]
and  unitary representations $v:G\to \m U \ma B^a (\m F)$ and
$w':G\to \m U \ma B^a ({E}\bigodot \m F)$ satisfying conditions
$(a)$-$(d)$ of the theorem. For all $l\in C_c (G,{E})$, $\sigma\in
S$ we obtain
\begin{eqnarray*}
 \tilde{\ma K}^{\sigma}(l) =\int_{G} {\ma K^{\sigma}}(l(t)) u_{t}dt = \int_{G}\nu(l(t)\odot\mf{i}(\sigma))u_t dt = \int_{G}\nu(l(t)\odot v_t \mf{i}(\sigma)) dt.
\end{eqnarray*}
Finally, it follows that $\{\tilde{\ma K}^{\sigma}\}_{\sigma\in S}$ is a
$\tilde{\mf{K}}$-family because for $\sigma,\sigma'\in S$ and $l,m \in C_c(G,{E})$ we have
\begin{eqnarray*}
 &&\langle \tilde{\ma K}^{\sigma}(l), \tilde{\ma K}^{\sigma'}(m)\rangle = \left<\int_{G}\nu(l(t)\odot v_t \mf{i}(\sigma)) dt,\int_{G}\nu(m(s)\odot v_s \mf{i}(\sigma')) ds\right>
\\&=& \int_{G}\int_{G}\left<  v_{t}\mf{i}(\sigma),  \pi(\langle l(t),m(ts)\rangle)v_{ts}\mf{i}(\sigma') \right>dt ds\\ &=&\left< \mf{i}(\sigma), \int_{G}\int_{G} v_{t^{-1}} \pi(\langle l(t),m(ts)\rangle)v_{ts}\mf{i}(\sigma') dt ds\right>\\&=& \left< \mf{i}(\sigma), \int_{G}\int_{G} \pi(\alpha^{\eta}_{t^{-1}}(\langle l(t),m(ts)\rangle))v_{s}\mf{i}(\sigma') dt ds\right>=\left<\mf{i}(\sigma), \int_{G} \pi(\langle l,m\rangle(s))v_s \mf{i}(\sigma') ds\right> \\ &=&\tilde{\mf{K}}^{\sigma,\sigma'}(\langle l,m \rangle).
\end{eqnarray*}

\end{proof}

\section{Characterizations of $\mf{K}$-families}\label{sec2}

Let ${E}$ be a Hilbert $C^*$-module over a $C^*$-algebra $\m B$. By
$M_n ({E})$ we denote the Hilbert $M_n (\m B)$-module where $M_n (\m
B)$-valued inner product is defined by
\[
 \langle [x_{ij}]^n_{i,j=1}, [x'_{ij}]^n_{i,j=1}\rangle:=\left[\sum^n_{k=1} \langle x_{ki},x'_{kj}\rangle\right]^n_{i,j=1}~\mbox{for all}~[x_{ij}]^n_{i,j=1},~[x'_{ij}]^n_{i,j=1}\in M_n ({E}).
\]

\begin{definition}
Let ${F}$ be a Hilbert $C^*$-module over a $C^*$-algebra $\m C$ and
let $T:{E}\to {F}$ be a linear map. For each positive integer $n$,
define $T_n: M_n({E})\to M_n ({F})$ by $$T_n
([x_{ij}]^n_{i,j=1}):=[T(x_{ij})]^n_{i,j=1}~\mbox{for
all}~[x_{ij}]^n_{i,j=1}\in M_n ({E}).$$ We say that $T$ is {\rm
completely bounded} if for each positive integer $n$, $T_n$ is
bounded and $\|T\|_{cb}:=\mbox{sup}_n \|T_n\|$ $<\infty$.
\end{definition}

We show in this section that $\mf{K}$-families, where $\mf{K}$ is a CPD-kernel,  are same as certain completely bounded maps between the Hilbert $C^*$-modules.
We need  the following  Hilbert $C^*$-modules to inspect the  extendability of $\mf{K}$- families to CPD-kernels between
the (extended) linking algebras of the Hilbert $C^*$-modules:

The vector space $E_n$  consists of elements
$(x_1,x_2,\ldots,x_n)$ with $x_i\in E$ for $1\leq i\leq n$, where the
operations are coordinate-wise. It becomes a Hilbert $M_n (\m
B)$-module with respect to the inner product whose $(i,j)$-entry is
given by
\[
 \langle (x_1,x_2,\ldots,x_n),(x'_1,x'_2,\ldots,x'_n)\rangle_{ij}:=\langle x_i,x'_j\rangle~\mbox{for}~(x_1,x_2,\ldots,x_n),(x'_1,x'_2,\ldots,x'_n)\in E_n.
\]
The symbol $E^n$ denotes the Hilbert $\m B$-module whose elements are
$(x_1,x_2,\ldots,x_n)^t$ with $x_i\in E$ for $1\leq i\leq n$, where
$^t$ denotes the transpose. The inner product in  $E^n$ is defined by
$$ \langle (x_1,x_2,\ldots,x_n)^t,(x'_1,x'_2,\ldots,x'_n)^t\rangle:=\sum^n_{i=1}\langle x_i,x'_i\rangle$$
$\mbox{for}~(x_1,x_2,\ldots,x_n)^t,(x'_1,x'_2,\ldots,x'_n)^t\in
E^n$.

From Lemma 3.2.1 of \cite{BBLS04} we know that $\mf{K}$ is a
CPD-kernel over $S$ from $\m B$ to $\m C$ if and only if for all
$\sigma_1,\sigma_2,\ldots,\sigma_n$ $(n\in\mathbb{N})$ the map $[
\mf{K}^{\sigma_i, \sigma_j}  ]^n_{i,j=1}:M_n (\m B)\to M_n (\m C)$
defined by $$[ \mf{K}^{\sigma_i, \sigma_j}  ][b_{ij}]:=[
\mf{K}^{\sigma_i, \sigma_j} (b_{ij})]^n_{i,j=1}~\mbox{ for
all}~[b_{ij}]^n_{i,j=1}\in M_n (\m B)$$ is (completely) positive. This realisation of
CPD-kernels comes in handy in the proof of the following theorem:

\begin{theorem}
 Let ${E}$ be a full Hilbert $C^*$-module over a $C^*$-algebra $\m B$ and let ${F}$ be a Hilbert $C^*$-module over a $C^*$-algebra $\m C$. 
Let $S$ be a set and let $\ma K^{\sigma}$ be a linear map from $E$ to $F$ for each $\sigma\in S$. Let $F_{\ma K}:=[\{\ma K^\sigma(x)c:x\in E,~c\in \m C,
~\sigma\in S\}]$. Then the following statements are equivalent:
 \begin{enumerate}
 \item [(a)] There exists unique CPD-kernel $\mf{K}:S\times S\to \ma B(\m B,\m C)$  such that $\{\ma K^{\sigma}\}_{\sigma\in S}$ is a $\mf{K}$-family.
  \item [(b)] $\{\ma K^{\sigma}\}_{\sigma\in S}$ extends to block-wise bounded linear maps $\begin{pmatrix}
{\mf{K}^{\sigma,\sigma'}} &  {\ma K^{\sigma^*}} \\
 \ma K^{\sigma'} & \vartheta
\end{pmatrix}$ from $\m L_{{E}}$ to $\m L_{{F}_{\ma K}}$ forming a CPD-kernel over $S$ from $\m L_{{E}}$ to $\m L_{{F}_{\ma K}}$, where $\vartheta$ is a $*$-homomorphism, i.e., $\{\ma K^{\sigma}\}_{\sigma\in S}$ is a CPD-H-extendable family.
  \item [(c)] For each finite choices $\sigma_1,\ldots,\sigma_n\in S$ the map from $E_n$ to $F_n$ defined by
$${\bf x}\mapsto({\ma K^{\sigma_1}}(x_1),{\ma K^{\sigma_2}}(x_2),\ldots,{\ma K^{\sigma_n}}(x_n))~\mbox{for}~{\bf x}=(x_1,x_2\ldots,x_n)\in E_n$$
is a completely bounded map. Moreover $F_{\ma K}$ can be made into a
$C^*$-correspondence from $\ma B^a ({E})$ to $\m C$ such that the
action of $\ma B^a ({E})$ on $F_{\ma K}$ is non-degenerate and for
each $\sigma\in S$, $\ma K^{\sigma}$ is a left $\ma B^a
({E})$-linear map.

   \item [(d)] For each finite choices $\sigma_1,\ldots,\sigma_n\in S$ the map from $E_n$ to $F_n$ defined by
$${\bf x}\mapsto({\ma K^{\sigma_1}}(x_1),{\ma K^{\sigma_2}}(x_2),\ldots,{\ma K^{\sigma_n}}(x_n))~\mbox{for}~{\bf x}=(x_1,x_2\ldots,x_n)\in E_n$$
is a completely bounded map and $\{\ma K^{\sigma}\}_{\sigma\in S}$
satisfies
   $$\langle \ma K^\sigma{( y)},\ma K^{\sigma'}({ x}\langle { x}',{ y}'\rangle)\rangle=\langle \ma K^{\sigma}({ x}'\langle {x},{ y}\rangle),\ma K^{\sigma'}({y}')\rangle~\mbox{for ${x},{ y},{ x}',{ y}'\in{E}$.}$$

\end{enumerate}

\end{theorem}

\begin{proof}
(a)$\Rightarrow$(b): Suppose $\m B$ is unital. Using Theorem
\ref{thm1} or Kolmogorov decomposition we get a pair $(\m F, \mf{i})$ consisting of a
$C^*$-correspondence $\m F$ from $\m B$ to $\m C$ and a map
$\mf{i}: S\to \m F$ such that $[\{b\mf{i}(\sigma)c:b\in\m B,c\in\m
C,\sigma\in S\}]=\m F$, and  an isometry $\nu:{E}
\bigodot\m F\to{F}$ defined by
 \[
  \nu(x\odot b\mf{i}(\sigma) c):={\ma K^\sigma}(xb)c~\mbox{for all}~x\in{E},~b\in\m B,~c\in\m C,~\sigma\in S.
 \]
We denote the unitary obtained from $\nu$, by restricting its  codomain to $F_{\ma K}$, with $\nu$ again. With this unitary $\nu$, 
define a $*$-homomorphism $\vartheta:\ma B^a ({E})\to \ma B^a
(F_{\ma K})$ by $\vartheta:a\mapsto \nu(a\odot id_{\m F})\nu^*$.
Identify $\m F$ with $\ma B^a (\m C,\m F)$ using $f\mapsto L_f$
where $L_f:c\mapsto fc$ and identify $\m B\bigodot \m F$ with $\m F$
using $b\odot f\mapsto bf$.  For each
$x,x'\in {E}$, $f$ and $f'\in \m F$, and $b\in \m B$ we obtain
\begin{align*}
 \langle (x\odot id_{\m F})^* (x'\odot f),b\odot f' \rangle &= \langle  x'\odot f,x b\odot f' \rangle=\langle  f,\langle x',x b\rangle  f' \rangle= \langle  f,\langle x',x\rangle bf' \rangle\\ &=\langle  x^*x'f, b f' \rangle=\langle  x^*x'\odot f, b\odot  f' \rangle\\ & =\langle  (x^*\odot id_{\m F}) ( x' \odot f), b\odot f' \rangle.
\end{align*}
Therefore
$(x\odot id_{\m F})^* =(x^*\odot id_{\m F}),~\mbox{for}~x\in {E}.$

For each $\sigma\in S$, the element $\begin{pmatrix}
{\mf{i}(\sigma)} &   \\
 & \nu^*
\end{pmatrix}\in \ma B^a\left(\begin{pmatrix}
{\m C}   \\
 {F_{\ma K}}
\end{pmatrix},\begin{pmatrix}
{\m B}    \\
 {E}
\end{pmatrix}\bigodot \m F\right)$. We have
 \begin{eqnarray*}
\begin{pmatrix}
{\mf{i}(\sigma)^*} &   \\
 & \nu
\end{pmatrix}\left(\begin{pmatrix}
{b} &  x^* \\
 y & a
\end{pmatrix}\bigodot id_{\m F}\right)\begin{pmatrix}
{\mf{i}(\sigma')} &   \\
 & \nu^*
\end{pmatrix}  &=&\begin{pmatrix}
{\mf{i}(\sigma)^*} &   \\
 & \nu
\end{pmatrix}\begin{pmatrix}
{b\odot \mf{i}(\sigma')} &  (x^*\odot id_{\m F})\nu^* \\
 y\odot \mf{i}(\sigma') & (a\odot id_{\m F})\nu^*
\end{pmatrix} \\&=& \begin{pmatrix}
{\mf{i}(\sigma)^*(b\odot \mf{i}(\sigma'))} &  \mf{i}(\sigma)^*(x\odot id_{\m F})^*\nu^* \\
 \nu(y\odot \mf{i}(\sigma')) & \nu(a\odot id_{\m F})\nu^*
\end{pmatrix}
 \end{eqnarray*}
for all $b\in \m B,x$ and $y\in{E}, a\in \ma B^a ({E}),$ and
$\sigma$ and $\sigma'\in S$. Thus, we get a CPD-kernel on $S$ from $\m
L_{{E}}$ to $\m L_{{F}_{\ma K}}$ formed by maps $\begin{pmatrix}
{\mf{K}^{\sigma,\sigma'}} &  {\ma K^{\sigma^*}} \\
 \ma K^{\sigma'} & \vartheta
\end{pmatrix}:=\begin{pmatrix}
{\mf{i}(\sigma)} &   \\
 & \nu^*
\end{pmatrix}^* (\bullet \odot id_{\m F})\begin{pmatrix}
{\mf{i}(\sigma')} &   \\
 & \nu^*
\end{pmatrix}$ where $\ma K^{\sigma^*} (x^*):=\ma K^\sigma (x)^*$ for $\sigma\in S,x \in E$.

Assume that $\m B$ is not unital. Let $\tilde{\m B}$ and $\tilde{\m
C}$ be the unitalizations of $\m B$ and $\m C$, respectively. Let
$(e_\lambda)_{\lambda\in \Lambda}$ be a contractive approximate unit
for $\m B$. Let $\delta:\tilde{\m B}\to \mathbb{C}$ be the unique
character vanishing on $\m B$. For each $\sigma,\sigma'$ define
$\tilde{\mf{K}}^{\sigma,\sigma'}:\tilde{\m B}\to\tilde{\m C}$ by
$\tilde{\mf{K}}^{\sigma,\sigma'}(b):={\mf{K}}^{\sigma,\sigma'}(b)$
for all $b\in \m B$ and
$\tilde{\mf{K}}^{\sigma,\sigma'}({1}_{\tilde{\m
B}}):=\|{\mf{K}}^{\sigma,\sigma'}\|{1}_{\tilde{\m C}}$. For each
$\lambda\in \Lambda$ define
${\mf{K}}^{\sigma,\sigma'}_\lambda:={\mf{K}}^{\sigma,\sigma'}(e^*_\lambda
\bullet e_\lambda)+(\|{\mf{K}}^{\sigma,\sigma'}\|{1}_{\tilde{\m
C}}-{\mf{K}}^{\sigma,\sigma'}(e^*_\lambda e_\lambda))\delta$. Mappings $\mf{K}_\lambda$s are  CPD-kernels and
$({\mf{K}}^{\sigma,\sigma'}_\lambda)_{\lambda\in \Lambda}$ converges pointwise to
$\tilde{\mf{K}}^{\sigma,\sigma'}$. We conclude that $\tilde{\mf{K}}
$ is a CPD-kernel. Note that $\{\ma K^\sigma\}_{\sigma\in S}$ is also a $\tilde{\mf{K}}$-family, and $E$ and $F$ are also Hilbert $C^*$-modules over $\tilde{\m B}$ 
and $\tilde{\m C}$, respectively. Extend $\{\ma K^\sigma\}_{\sigma\in S}$ to a CPD-kernel 
over $S$ from $
\begin{pmatrix}
{\tilde{\m B}} & {{{E}}^*}  \\
{{E}} & \ma B^a ({E})
\end{pmatrix}$ to $\m L_{{F}_{\ma K}}$ as above. Restricting this CPD-kernel to  $
\begin{pmatrix}
{\m B} & {{{E}}^*}  \\
{{E}} & \ma B^a ({E})
\end{pmatrix}$  yields the required CPD-kernel.  \\(b)$\Rightarrow$(c): Let $n\in \mathbb{N}$. For
$\sigma_1,\ldots,\sigma_n\in S$ define a linear map ${\bf K}$ from
$E_n$ to $F_n$ by
$${\bf x}\mapsto({\ma K^{\sigma_1}}(x_1),{\ma K^{\sigma_2}}(x_2),\ldots,{\ma K^{\sigma_n}}(x_n))~\mbox{for}~{\bf x}=(x_1,x_2\ldots,x_n)\in E_n.$$ Fix $l\in \mathbb{N}$ and let $[{\bf x}_{ms}]^l_{m,s=1}\in M_l (E_n)$ where ${\bf x}_{ms}$$=(x_{ms,1},x_{ms,2}$$,\ldots,x_{ms,n}) \in E_n.$ Set
\begin{align*}
A:= \begin{bmatrix}
  \bigl(\begin{smallmatrix} 0 & 0 \\ a_1 & 0 \end{smallmatrix}\bigr) & \bigl(\begin{smallmatrix} 0 & 0 \\ a_2 & 0  \end{smallmatrix}\bigr)  & \ldots&\bigl(\begin{smallmatrix}  0 & 0\\ a_n & 0 \end{smallmatrix}\bigr) \\
\bigl(\begin{smallmatrix} 0 & 0 \\ 0 & 0 \end{smallmatrix}\bigr) & \bigl(\begin{smallmatrix} 0 & 0 \\ 0 & 0  \end{smallmatrix}\bigr)  & \ldots&\bigl(\begin{smallmatrix}  0 & 0\\ 0 & 0 \end{smallmatrix}\bigr)  \\
\vdots& \vdots & & \vdots \\ \bigl(\begin{smallmatrix} 0 & 0 \\ 0 & 0
\end{smallmatrix}\bigr) & \bigl(\begin{smallmatrix} 0 & 0 \\ 0 & 0  \end{smallmatrix}\bigr)  &
\ldots&\bigl(\begin{smallmatrix}  0 & 0\\ 0 & 0
\end{smallmatrix}\bigr)
\end{bmatrix}.
\end{align*}
Define $B_{mk}$ and $C_{mk}$ as the matrix $A$ where  $a_i=\ma K^{\sigma_i}(x_{mk,i})$ and $a_i=x_{mk,i}$ respectively. We have
\begin{align*}
& \|{\bf K}_l ([{\bf x}_{ms}]^l_{m,s=1})\|^2 = \| [{\bf K}({\bf
x}_{ms})]^l_{m,s=1}\|^2=\|\langle [{\bf K}({\bf
x}_{ms})]^l_{m,s=1},[{\bf K}({\bf x}_{ms})]^l_{m,s=1}
\rangle\|\\&=\left\| \left[\sum^l_{k=1} \langle {\bf K}({\bf
x}_{km}),{\bf K}({\bf x}_{ks}) \rangle\right]^l_{m,s=1}\right\|
=\left\| \left[\sum^l_{k=1} \left[\langle \ma
K^{\sigma_i}(x_{km,i}),\ma K^{\sigma_j}(x_{ks,j})
\rangle\right]^n_{i,j=1}\right]^l_{m,s=1}\right\|
\\&=\| [\sum^l_{k=1}
 B_{km}^*B_{ks}
 ]^l_{m,s=1} \|=\| [
B_{ms} ]^l_{m,s=1} \|^2
\\&=\left\| \left[
\left[\begin{pmatrix}
{\mf{K}^{\sigma_i,\sigma_j}} &  {\ma K^{\sigma_i^*}} \\
 \ma K^{\sigma_j} & \vartheta
\end{pmatrix} \right]
  C_{ms} \right]^l_{m,s=1} \right\|^2 \leq\left\|\left[\begin{pmatrix}
{\mf{K}^{\sigma_i,\sigma_j}} &  {\ma K^{\sigma_i^*}} \\
 \ma K^{\sigma_j} & \vartheta
\end{pmatrix} \right]_l\right\|^2 \left\|\left[{\bf x}_{ms}\right]^l_{m,s=1}\right\|^2
\end{align*}
where $2\times2$ matrices with round brackets are elements of the
linking algebras. Therefore from Lemma 3.2.1 of \cite{BBLS04} it
follows that ${\bf K}$ is completely bounded.  Let $\m
D:=\begin{pmatrix} 0&0\\0&\ma B^a ({E})\end{pmatrix}$ be a
$C^*$-subalgebra of $\m L_{{E}}$ with the unit $1_{\m
D}:=\begin{pmatrix} 0&0\\0&id_{{E}}\end{pmatrix}$. We denote by $\theta$ the $*$-homomorphism which is the
restriction of $\begin{pmatrix}
{\mf{K}^{\sigma,\sigma'}} &  {\ma K^{\sigma^*}} \\
 \ma K^{\sigma'} & \vartheta
\end{pmatrix}$ to $\m D$. Without loss of generality we assume that $\m B$ is unital because if $\m B$ is not unital, then we can unitalize it and 
work as in the proof of ``(a) $\Rightarrow$ (b)''. Let $(\m F,\mf{i})$ be the Kolmogorov
decomposition for the CPD-kernel  $\begin{pmatrix}
{\mf{K}^{\sigma,\sigma'}} &  {\ma K^{\sigma^*}} \\
 \ma K^{\sigma'} & \vartheta
\end{pmatrix}$ where $\sigma,\sigma'\in S$. For
each $d\in\m D$ and $\sigma\in S$,
\begin{align*}&\|d \mf{i}(\sigma)-1_{\m D}\mf{i}(\sigma)\theta(d)\|^2
\\=&\|\langle d\mf{i}(\sigma),d\mf{i}(\sigma)\rangle-\langle d\mf{i}(\sigma),1_{\m D}\mf{i}(\sigma)\theta(d)\rangle-\langle 1_{\m D}\mf{i}(\sigma)\theta(d),d\mf{i}(\sigma)\rangle+\langle 1_{\m D}\mf{i}(\sigma)\theta(d),1_{\m D}\mf{i}(\sigma)\theta(d)\rangle\|
\\=&\|\theta(d^*d)-\theta(d^*d)-\theta(d^*d)+\theta(d^*d)\|=0.\end{align*}
Therefore for each $\sigma,\sigma'\in S$ and for all $x\in{E}$,
$a\in \ma B^a ({E})$ we have
\begin{align*}
\begin{pmatrix}
0 &  0 \\
  \ma K^{\sigma'}(ax) & 0
\end{pmatrix} &=\begin{pmatrix}
{\mf{K}^{\sigma,\sigma'}} &  {\ma K^{\sigma^*}} \\
 \ma K^{\sigma'} & \vartheta
\end{pmatrix}\left(
\begin{pmatrix}
0 &  0 \\
 0 & a
\end{pmatrix}
\begin{pmatrix}
0 &  0 \\
 x & 0
\end{pmatrix}\right)=\left\langle \mf{i}(\sigma) ,\left(
\begin{pmatrix}
0 &  0 \\
 0 & a
\end{pmatrix}
\begin{pmatrix}
0 &  0 \\
 x & 0
\end{pmatrix}\right)\mf{i}(\sigma')\right\rangle
\\&=\left\langle \left(\begin{pmatrix}
0 &  0 \\
 0 & a
\end{pmatrix}^*\right)\mf{i}(\sigma) ,\left(
\begin{pmatrix}
0 &  0 \\
 x & 0
\end{pmatrix}\right)\mf{i}(\sigma')\right\rangle
\\ &=\left\langle 1_{\m D}\mf{i}(\sigma) \theta\left(\begin{pmatrix}
0 &  0 \\
 0 & a
\end{pmatrix}^*\right) ,\left(\begin{pmatrix}
0 &  0 \\
 x & 0
\end{pmatrix}\right)\mf{i}(\sigma')\right\rangle
\end{align*}
\begin{align*}  &=\begin{pmatrix}
0 &  0 \\
 0 & \vartheta(a)
\end{pmatrix}\begin{pmatrix}
{\mf{K}^{\sigma,\sigma'}} &  {\ma K^{\sigma^*}} \\
 \ma K^{\sigma'} & \vartheta
\end{pmatrix}\left(\begin{pmatrix}
0 &  0 \\
 x & 0
\end{pmatrix}\right)=\begin{pmatrix}
0 &  0 \\
 \vartheta(a) \ma K^{\sigma'}(x) & 0
\end{pmatrix}.
\end{align*}
Hence $\ma K^{\sigma'}$ is
a left $\ma B^a ({E})$-linear map for each $\sigma'\in S$ and $\vartheta$ is non-degenerate. Observe that the Hilbert $C^*$-module $F_{\ma K}$ is a 
$C^*$-correspondence from $\ma B^a(E)$ to $\m C$  with the left action is given by $\vartheta$.\\
(c)$\Leftrightarrow$(d): If $\ma K^{\sigma}$ is a left $\ma B^a
({E})$-linear map for each $\sigma\in S$, then
   \begin{align*}\langle\ma K^{\sigma}( y),\ma K^{\sigma'}({x}\langle { x}',{ y}'\rangle) &=\langle{\ma K^{\sigma}}({ y}),{\ma K^{\sigma'}}({ x}~{ x}^{\prime*}{ y}')
\rangle=\langle({ x}~{ x}^{\prime*})^*{\ma K^{\sigma}}({ y}),{\ma K^{\sigma'}}({ y}')\rangle\\ &=\langle {\ma K^{\sigma}}({ x}' { x}^*{ y}),\ma K^{\sigma'}({ y}')
\rangle =\langle {\ma K^{\sigma}}({ x}'\langle { x},{ y}\rangle),\ma K^{\sigma'}({ y}')\rangle\end{align*} for all ${ x},{ y},{ x}',{ y}'\in{E}$ and 
$\sigma,\sigma'\in S$. Conversely using the equation in condition (d), we  define an action  $\vartheta$ on $F_{\ma K}$, of the algebra $\ma F({E})$ of all finite rank operators on ${E}$, by
  $$\vartheta ({x}'{x}^*){\ma K^\sigma}({ y}):={\ma K^\sigma}({ x}'{ x}^*{ y})~\mbox{for all ${x},{ x}',{ y}\in{E}$.}~$$
 Since $\vartheta$ is bounded on $\ma F({E})$, it extends naturally as an adjointable action of $\sK({E})$ on $F_{\ma K}$. Since ${E}$ is full, we can obtain 
an approximate unit $\left(\sum^{k_\lambda}_{n=1} \langle { x}^\lambda_n , {y}^\lambda_n\rangle\right)_{\!\!\lambda\in \Lambda}$ for $\m B$ where ${ x}^\lambda_n,~ { y}^\lambda_n\in {E}$. Using this approximate unit, it follows that $\vartheta$ is non-degenerate. We can further extend this action to an action of $\ma B^a ({E})$ on $F_{\ma K}$ (cf. Proposition 2.1 of \cite{La95}).
\\ (c)$\Rightarrow$(a): Let $n\in \mathbb{N}$. The algebraic tensor product ${{E}_n}^*\underline{\bigodot}
{E}_n$=
 span$\langle {E}_n,{E}_n\rangle$ (cf. Proposition 4.5 of \cite{La95}). Note that ${{E}_n}^*\underline{\bigodot}
{E}_n$ is a dense subset of $M_n (\m
 B)$. Set $\sigma_1,\ldots,\sigma_n\in S$ and let ${\bf K}$ be defined as above. For each $k\in\mathbb{N}$ we define ${\bf{K}}^k:({E_n})^{k}\to ({F_n})^{k}$ by
\[
{\bf{K}}^k({\bf x}^k):=({\bf{K}}({\bf x}_1),{\bf{K}}({\bf
x}_2),\ldots,{\bf{K}}({\bf x}_k))^t~\mbox{where}~
 {\bf x}^k=({\bf x}_1,{\bf x}_2,\ldots,{\bf x}_k)^t\in
 ({E_n})^{k}.
\]
 Define a linear map $[\mf K^{\sigma_i, \sigma_j}]^n_{i,j=1}:{{E}_n}^*\underline{\bigodot}
{E}_n\to M_n (\m
 C)$ by
 \[
 [\mf K^{\sigma_i, \sigma_j}] \left(\sum^k_{l=1} \langle
 {\bf x}_l,{\bf y}_l\rangle\right):=\langle {\bf{K}}^k({\bf x}^k),{\bf{K}}^k({\bf y}^k)\rangle
 \]
 where
 ${\bf x}^k=({\bf x}_1,{\bf x}_2,\ldots,{\bf x}_k)^t$, ${\bf y}^k=({\bf y}_1,{\bf y}_2,\ldots,{\bf y}_k)^t\in ({E_n})^{k}$ (i.e., $\langle{\bf x}^k,{\bf y}^k\rangle=\sum^k_{i=1} \langle{\bf x}_i,{\bf y}_i\rangle$). First, we prove that $[\mf K^{\sigma_i, \sigma_j}]$ is bounded. We have
 \begin{align*}
\left\|[\mf K^{\sigma_i, \sigma_j}] \left(\sum^k_{l=1} \langle {\bf
x}_l,{\bf y}_l\rangle\right)\right\| =\|\langle
 {\bf {K}}^k({\bf x}^k),{\bf {K}}^k({\bf y}^k)\rangle\|\leq
  \|{\bf{K}}\|^2_{cb} \|{\bf x}^k\|\|{\bf y}^k\|.
 \end{align*}
For $0<\alpha<1$ we  decompose ${\bf x}^{k*}$ as ${\bf
w}^k_{\alpha}|{\bf x}^{k*}|^{\alpha}$ (cf. Lemma 4.4 of \cite{La95};
Lemma 2.9 of \cite{SSu14}) where ${\bf w}^k_{\alpha}:=|{\bf
x}^{k*}|^{1-\alpha}$. So as $\alpha \to 1$ we have
\begin{align*}\left\|\displaystyle\sum^k_{l=1} \langle
 {\bf x}_l,{\bf y}_l\rangle\right\|=&\|\langle{\bf x}^k,{\bf y}^k\rangle\|=\|{\bf x}^{k*}\odot{\bf y}^k\|=\|{\bf w}^k_{\alpha}|{\bf x}^{k*}|^{\alpha}
\odot{\bf y}^k\|=\|{\bf w}^k_{\alpha}\odot|{\bf x}^{k*}|^{\alpha}{\bf y}^k\|\\ \leq&\|{\bf w}^k_{\alpha}\|\||{\bf x}^{k*}|^{\alpha}{\bf y}^k\|\to\||
{\bf x}^{k*}|{\bf y}^k\|=\|\langle{\bf x}^k,{\bf y}^k\rangle\|.
\end{align*}
In the above equation array we have used the facts that $\|{\bf w}^k_{\alpha}\|=sup _{\lambda\in \sigma(|{\bf x}^{k*}|)}
\lambda^{1-\alpha}=\|{\bf x}^{k*}\|^{1-\alpha}\to
 1$ and  $|{\bf x}^{k*}|^{\alpha}$ converges in norm to
 $|{\bf x}^{k*}|$. We deduce that for each $\epsilon>0$ there exists $\alpha$ such that $$\|{\bf w}^k_{\alpha}\|\||{\bf x}^{k*}|^{\alpha}{\bf y}^k\|\leq \left\|\displaystyle\sum^k_{l=1} \langle
 {\bf x}_l,{\bf y}_l\rangle\right\|+\epsilon.$$
Let ${\bf x}'^{k}:={\bf w}^{k*}_{\alpha}\in ({E_n})^{k}$ and ${\bf
y}'^k=|{\bf x}^{k*}|^{\alpha}{\bf y}^k\in ({E_n})^{k}$. Then
$\|\langle{\bf x}'^{k},{\bf y}'^{k}\rangle\|\leq\|{\bf x}'^k\|\|{\bf
y}'^k\|\leq \left\|\displaystyle\sum^k_{l=1} \langle
 {\bf x}_l,{\bf y}_l\rangle\right\|+\epsilon$ and
$$\langle{\bf x}'^{k},{\bf y}'^{k}\rangle={\bf x}'^{k*}\odot{\bf y}'^{k}={\bf x}'^{k*}\odot{\bf y}'^{k}={\bf w}^k_{\alpha}\odot|{\bf x}^{k*}|^{\alpha}{\bf y}^k={\bf w}^k_{\alpha}|{\bf x}^{k*}|^{\alpha}\odot{\bf y}^k=\langle{\bf x}^k,{\bf y}^k\rangle.$$
Therefore $[\mf K^{\sigma_i, \sigma_j}]$ is bounded.

Because $E_n$ is full, as in the case $(c)\Leftrightarrow(d)$, we can get the
approximate unit $e_\lambda=\langle {\bf X}_\lambda,{\bf
Y}_\lambda\rangle$ for $M_n (\m B)$  where
${\bf X}_\lambda=({\bf x}^{\lambda}_1,{\bf
x}^{\lambda}_2,\ldots,{\bf x}^{\lambda}_{k_\lambda})^t,~{\bf
Y}_\lambda=({\bf y}^{\lambda}_1,{\bf y}^{\lambda}_2,\ldots,{\bf
y}^{\lambda}_{k_\lambda})^t\in (E_n)^{k_\lambda}$.  Let $B$ be a
positive elements in $M_n(\m B)$ and let $t_\lambda$ be the positive
square root of the rank one operator ${\bf X}_\lambda B {\bf
X}^*_\lambda$ in $\ma K((E_n)^{k_\lambda})$. Finally, using $e^*_\lambda B
e_\lambda\stackrel{\lambda}{\to} B$ in norm and
 \begin{align*}
  [\mf K^{\sigma_i, \sigma_j}](e^*_\lambda B e_\lambda)=&[\mf K^{\sigma_i, \sigma_j}]({\bf Y}^*_\lambda{\bf X}_\lambda B {\bf X}^*_\lambda {\bf Y}_\lambda)=[\mf K^{\sigma_i, \sigma_j}](\langle t_{\lambda}{\bf Y}_\lambda,t_{\lambda}{\bf Y}_\lambda\rangle)\\=&\langle {\bf K}^{k_\lambda} (t_{\lambda}{\bf Y}_\lambda), {\bf K}^{k_\lambda} (t_{\lambda}{\bf X}_\lambda)\rangle\geq 0,
 \end{align*}
we infer that $[\mf K^{\sigma_i, \sigma_j}](B)\geq 0$.
\end{proof}

Let $G$ be a locally compact group. Suppose $E$ is a full Hilbert
$C^*$-module over a unital $C^*$-algebra $\m B$ and
$(G,\eta,{E})$ is a dynamical system on ${E}$. We  define a
$C^*$-dynamical system on the linking algebra $\m L_{{E}}$ as follows: For each
$s\in G$, let us define Ad$\eta_s (a):=\eta_s a \eta_{s^{-1}}$ for
$a\in \ma B^a ({E})$ and define ${\eta^*_s}({x}^*):={\eta_s(x)}^*$
for $x\in{E}$. Denote by $\theta$ the action of $G$ on $\m
L_{{E}}$ which is given by
\begin{eqnarray*}
 \theta_s \left(\begin{pmatrix}
{b} & {{x}^*}  \\
{y} & a
\end{pmatrix}\right):=\begin{pmatrix}
{\alpha^\eta_s(b)} & {{\eta^*_s}({x}^*)}  \\
{\eta_s(y)} & Ad\eta_s a
\end{pmatrix}
\end{eqnarray*}
for all $s\in G$, $a\in \ma B^a ({E})$, $b\in \m B$ and $x,y\in{E}$.
It is easy to check that we obtain a  $C^*$-dynamical system $(G,\theta,\m L_{{E}})$.

\begin{theorem}
 Let ${E}$ be a full Hilbert $C^*$-module over a unital $C^*$-algebra $\m B$ and let ${F}$ be a $C^*$-correspondence from $\m D$ to $\m C$ where $\m C$ and 
$\m D$ are unital $C^*$-algebras. Let $u:G\to\m U\m C$, $u':G\to \m U\m D$ be unitary representations of a locally compact group $G$ and let $(G,\eta,{E})$ be 
a dynamical system on ${E}$. Assume $S$ to be a set and  $\ma K^{\sigma}$ to be a linear map from $E$ to $F$ for each $\sigma\in S$. Let $F_{\ma K}:=[\{\ma K^\sigma(x)c:x\in E,~c\in \m C,~\sigma\in S\}]$. Then the following statements are equivalent:
 \begin{enumerate}
 \item [(a)] There exists unique CPD-kernel $\mf{K}:S\times S\to \ma B(\m B,\m C)$  such that $\{\ma K^{\sigma}\}_{\sigma\in S}$ is a $(u',u)$-covariant $\mf{K}$-family with respect to the
dynamical system $(G,\eta,{E})$.
  \item [(b)] $\{\ma K^{\sigma}\}_{\sigma\in S}$ extends to block-wise bounded linear maps $\begin{pmatrix}
{\mf{K}^{\sigma,\sigma'}} &  {\ma K^{\sigma^*}} \\
 \ma K^{\sigma'} & \vartheta
\end{pmatrix}$ from $\m L_{{E}}$ to $\m L_{{F}_{\ma K}}$ forming a CPD-kernel over $S$ from $\m L_{{E}}$ to $\m L_{{F}_{\ma K}}$, where $\vartheta$ is a $*$-homomorphism, i.e., $\{\ma K^{\sigma}\}_{\sigma\in S}$ is a CPD-H-extendable family. The family is $\omega$-covariant with respect to $(G,\theta,\m L_{{E}})$ where $\omega: G\to \m U \m L_{{F}_{\ma K}}$ is a unitary representation.
  \item [(c)] For each finite choices $\sigma_1,\ldots,\sigma_n\in S$ the map from $E_n$ to $F_n$ defined by
$${\bf x}\mapsto({\ma K^{\sigma_1}}(x_1),{\ma K^{\sigma_2}}(x_2),\ldots,{\ma K^{\sigma_n}}(x_n))~\mbox{for}~{\bf x}=(x_1,x_2\ldots,x_n)\in E_n$$
is a completely bounded map. Moreover $\{\ma K^{\sigma}\}_{\sigma\in
S}$ is $(u^{\prime},u)$-covariant with respect to $(G,\eta,{E})$,
$F_{\ma K}$ is a correspondence from $\ma B^a ({E})$ to $\m C$
such that the action of $\ma B^a ({E})$ on $F_{\ma K}$ is
non-degenerate and for each $\sigma\in S$, $\ma K^{\sigma}$ is a left
$\ma B^a ({E})$-linear map.

   \item [(d)] For each finite choices $\sigma_1,\ldots,\sigma_n\in S$ the map from $E_n$ to $F_n$ defined by
$${\bf x}\mapsto({\ma K^{\sigma_1}}(x_1),{\ma K^{\sigma_2}}(x_2),\ldots,{\ma K^{\sigma_n}}(x_n))~\mbox{for}~{\bf x}=(x_1,x_2\ldots,x_n)\in E_n$$
is a completely bounded map and $\{\ma K^{\sigma}\}_{\sigma\in S}$
is $(u',u)$-covariant with respect to $(G,\eta,{E})$ satisfying
   $$\langle \ma K^\sigma{( y)},\ma K^{\sigma'}({ x}\langle { x}',{ y}'\rangle)\rangle=\langle \ma K^{\sigma}({ x}'\langle {x},{ y}\rangle),\ma K^{\sigma'}({y}')\rangle ~\mbox{for}~{x},{ y},{ x}',{ y}'\in{E}.$$

\end{enumerate}

\end{theorem}
\begin{proof}
We use the same notations as in the proof of part (a)$\Rightarrow$(b) of
the previous theorem. For each $s\in G$ define a map $\omega_s:  \m
L_{{F}}\to \m L_{{F}}$ by $$\omega_s \left(\begin{pmatrix}
{c} & {{x}^*}  \\
{y} & a
\end{pmatrix}\right):=\begin{pmatrix}
{u_s c} & {u_s {x}^*}  \\
{u'_s y} & u'_s a
\end{pmatrix}$$
 for all $c\in \m C$, $x,y \in {F}$ and $a\in \ma B^a ({F})$. The mapping $\omega: G\to \m U \m L_{{F}}$ is a unitary representation. Using Theorem \ref{thm2} we obtain a unitary representation $w':G\to \m U \ma B^a ({E}\bigodot \m F)$ defined by
$$w'_t(x\odot b\mf{i}(\sigma) c):=\eta_t(x)\odot v_t (b\mf{i}(\sigma) c)$$ for all $b\in\m B$, $c\in \m C$, $x\in{E}$, $\sigma\in S$ and $t\in G$. Further it satisfies $\nu w'_t=u'_t\nu$ for all $t\in G.$ Thus we have
\[
 \vartheta (\eta_s a \eta_{s^{-1}})=\nu((\eta_s a \eta_{s^{-1}})\odot id_{\m F})\nu^*=\nu w'_s (a \odot id_{\m F})w'_{s^{-1}} \nu^*=u'_{s}  \vartheta(a) u'_{s^{-1}}
\]
for all $s\in G$ and $a\in \ma B^a ({E})$. Therefore
\begin{eqnarray*}
 \begin{pmatrix}
{\mf{K}^{\sigma,\sigma'}} &  {\ma K^{\sigma^*}} \\
 \ma K^{\sigma'} & \vartheta
\end{pmatrix}\left(
 \theta_s \left(\begin{pmatrix}
{b} & {x^*}  \\
{y} & a
\end{pmatrix}\right)\right)&=& \begin{pmatrix}
{\mf{K}^{\sigma,\sigma'}(\alpha^\eta_s(b))} &  {\ma K^{\sigma^*}({\eta^*_s}(x^*))} \\
 \ma K^{\sigma'}(\eta_s(y)) & \vartheta(Ad\eta_s a)
\end{pmatrix}\\&=&\omega_s \begin{pmatrix}
{\mf{K}^{\sigma,\sigma'}} &  {\ma K^{\sigma^*}} \\
 \ma K^{\sigma'} & \vartheta
\end{pmatrix} \left(\begin{pmatrix}
{b} & {x^*}  \\
{y} & a
\end{pmatrix}\right)\omega^*_s
\end{eqnarray*}
for all $s\in G$, $a\in \ma B^a ({E})$, $b\in \m B$,
$\sigma,\sigma'\in S$ and $x,y\in{E}$.
\end{proof}

\section{Application to the dilation theory of CPD-kernels}\label{sec3}

Suppose ${E}$ and ${F}$ are Hilbert $C^*$-modules over $C^*$-algebras
$\m B$ and $\m C$ respectively. Let $S$ be a set and let
$\mf{K}:S\times S\to \ma B(\m B,\m C)$ be a CPD-kernel. Let $\{\ma
K^{\sigma}\}_{\sigma\in S}$ be a $\mf{K}$-family where ${\ma
K}^\sigma$ is a map from ${E}$ to ${F}$ for each $\sigma\in S$. Recall that
there exists the Kolmogorov decomposition $(\m F,\mf{i})$ of $\mf{K}$. From
Theorem \ref{thm1} it follows that there is an isometry $\nu:{E}\bigodot\m F\to{F}$
such that
\[
  \nu(x\odot \mf{i}(\sigma))=\ma{K^{\sigma}}(x)~\mbox{for all}~x\in{E},~\sigma\in S.
 \]
 If $F_{\ma K}$ is complemented in $F$, then we obtain a $*$-homomorphism $\vartheta$ from
$\ma B^a (E)$ to $\ma B^a (F)$ defined by $\nu(\bullet \odot id_{\m
F})\nu^{*}$. Also, if $\xi$ is a unit vector in $E$, i.e., $\langle
\xi,\xi\rangle=1$, then the following diagram commutes.
\begin{eqnarray}\label{diag6}
\xymatrix{ \m B \ar@{->}[r]^{ \mf{K}^{\sigma,\sigma'}}
     \ar@{->}[d]_{\xi\bullet \xi^*}
&\m C
     \ar@{<-}[d]^{\langle \nu(\xi\odot \mf{i}(\sigma)),\bullet \nu(\xi\odot \mf{i}(\sigma'))\rangle}
\\
\ma B^a (E) \ar@{->}[r]^{\vartheta}
     & \ma B^a (F)
}
\end{eqnarray}

\noindent Here $b\mapsto \xi b\xi^*$ is a representation of $\m B$ on $E$.
In fact, to obtain the above commuting diagram, it is sufficient to assume that there exist a $C^*$-correspondence $\m F$ from $\m B$ to $\m C$,
 a map $\mf{i}: S\to \m F$, a Hilbert $\m B$-module $E$, an adjointable isometry $\nu:E\bigodot \m F\to F$ and a unit vector $\xi\in E$. For this we set $\mf{K}^{\sigma,\sigma'}:=\langle \mf{i}(\sigma),\bullet
\mf{i}(\sigma')\rangle$ for $\sigma,\sigma'\in S$ and $\vartheta:=\nu(\bullet \odot id_{\m F})\nu^{*}$.

If $\mf{i}(\sigma)$'s are also unit vectors, then $\mf{K}^{\sigma,\sigma'}$ is a unital map for each $\sigma,\sigma'\in S$, and in this case  we say that kernel
$\mf{K}$ is {\it Markov}  and the dilation $\vartheta$ of $\mf{K}$ is a {\it weak dilation}.
Change the map $\xi \bullet \xi^*$ by the map $\langle \xi, \bullet
\xi\rangle$ and reverse the arrow of this map. Now substitute $\ma
K^\sigma (\xi)=\nu (\xi\odot \mf{i}(\sigma))$ in the above diagram
to get the commuting diagram:

\begin{eqnarray}\label{diag7}
\xymatrix{ \m B \ar@{->}[r]^{ \mf{K}^{\sigma,\sigma'}}
     \ar@{<-}[d]_{\langle \xi,\bullet \xi\rangle}
&\m C
     \ar@{<-}[d]^{\langle \ma K^{\sigma} (\xi),\bullet \ma K^{{\sigma}'}(\xi)\rangle}
\\
\ma B^a (E) \ar@{->}[r]^{\vartheta}
     & \ma B^a (F)
}
\end{eqnarray}

\noindent This motivates us to introduce a notion of dilation of a CPD-kernel $\mf{K}$ over $S$ whenever there is a family of maps  
$\{\ma K^{\sigma}\}_{\sigma\in S}$ between some Hilbert $C^*$-modules
and there is a similar commuting diagram as above.

\begin{definition}
  Let ${E}$ and ${F}$ be Hilbert $C^*$-modules over $C^*$- algebras $\m B$ and $\m C$ respectively. Let $S$ be a set and let $\mf{K}:S\times S\to \ma B(\m B,\m C)$ 
be a CPD-kernel. A $*$-homomorphism $\vartheta: \ma B^a (E)\to \ma B^a (F)$ is a {\rm CPDH-dilation} of $\mf{K}$ if $E$ is full and if there is a linear map 
${\ma K}^\sigma$ from ${E}$ to ${F}$ for each $\sigma\in S$ such that
\begin{eqnarray}\label{diag1}
\xymatrix{ \m B \ar@{->}[r]^{ \mf{K}^{\sigma,\sigma'}}
     \ar@{<-}[d]_{\langle x,\bullet x'\rangle}
&\m C
     \ar@{<-}[d]^{\langle \ma K^{\sigma} (x),\bullet \ma K^{{\sigma}'}(x')\rangle}
\\
\ma B^a (E) \ar@{->}[r]^{\vartheta}
     & \ma B^a (F)
}
\end{eqnarray}
commutes for all $x,x'\in E$. The CPDH-dilation $\vartheta$ is called
\begin{itemize}
 \item [(a)] {\rm quasi-dilation} if $E$ is not
necessarily full.
\item [(b)] {\rm strict} if the $*$-homomorphism $\vartheta$ is strict.
\item [(c)] {\rm CPDH$_0$-dilation} if $\vartheta$ is a unital $*$-homomorphism.
\end{itemize}
\end{definition}

\begin{proposition}\label{prop1}
 Let $\vartheta$ be a CPDH$_0$-quasi-dilation of a CPD-kernel $\mf{K}:S\times S\to \ma B(\m B,\m C)$.  If $\{\ma K^{\sigma}\}_{\sigma\in S}$ is a family of maps from $E$ to $F$ such that the Diagram \ref{diag1} commutes, then $\{\ma K^{\sigma}\}_{\sigma\in S}$ is a $\mf{K}$-family where $$\ma K^{\sigma} (ax)=\vartheta(a) \ma K^{\sigma} (x)~\mbox{for}~x\in E,~a\in \ma B^a (E),~\sigma\in S.$$
\end{proposition}
\begin{proof}
Since the Diagram \ref{diag1} commutes, for $x\in E$, $a\in \ma B^a (E)$ and $\sigma,\sigma'\in S$ we get
\begin{eqnarray}\label{eqn**}
\langle \ma{K}^\sigma (x),\vartheta(a) \ma{K}^{\sigma'} (x')\rangle=\langle \ma{K}^\sigma (x),\ma{K}^{\sigma'} (ax')\rangle.
 \end{eqnarray}
 As $\vartheta$ is unital, $\{\ma K^{\sigma}\}_{\sigma\in S}$ is a $\mf{K}$-family. So by setting
 $F_{\ma K}:=[\{\ma K^\sigma(e)c:e\in E,~c\in \m C,~\sigma\in S\}]$ and using equation \ref{eqn**} we get a $*$-homomorphism
$\vartheta_{\ma K}:\ma B^a (E)\to \ma B^a ({F}_{\ma K})$ which is defined by
$\vartheta_{\ma K}(a) \ma{K}^\sigma (x)=\ma{K}^\sigma (ax)~\mbox{for}~x\in E,~a\in \ma B^a (E),~\sigma,\sigma'\in S.$
We get $$\langle y,\vartheta_{\ma K}(a) y'\rangle=\langle y,\vartheta(a) y'\rangle~\mbox{for all}~a\in \ma B^a (E)~\mbox{and}~y,y'\in {F}_{\ma K}.$$
Thus, $\vartheta(a)y=\vartheta_{\ma K}(a)y$ for all $y\in {F}_{\ma K}$ and $a\in \ma B^a (E)$.
\end{proof}

\begin{definition}
 Let $\mf{K}:S\times S\to \ma B(\m B,\m C)$ be a CPD-kernel. A  family of maps $\{\ma K^{\sigma}\}_{\sigma\in S}$
from $E$ to $F$ is called {\rm (strict) CPDH$_0$-family},
if it extends as a CPD-kernel over $S$ from $\m L_{{E}}$ to $\m
L_{{F}}$  whose $(2,2)$-corner is a unital (strict)
$*$-homomorphism.
\end{definition}

\begin{proposition}
Let $\m B$ be unital. If $\vartheta$ is a strict
CPDH$_0$-dilation of a CPD-kernel $\mf{K}:S\times S\to \ma B(\m B,\m
C)$ and  $\{\ma K^{\sigma}\}_{\sigma\in S}$ is a family of maps
from $E$ to $F$ such that the Diagram \ref{diag1} commutes, then
$\{\ma K^{\sigma}\}_{\sigma\in S}$ is a strict CPDH$_0$-family.
\end{proposition}
\begin{proof}
 Let $(\m F_{\mf{K}}, \mf{i})$ be the Kolmogorov decomposition of the CPD-kernel $\mf{K}:S\times S\to \ma B(\m B,\m C)$. 
Because $\vartheta$ is a strict unital homomorphism from $\ma B^a (E)$ into $\ma B^a (F)$ using the representation theorem (Theorem 1.4) of \cite{MSS06}, 
we obtain a $C^*$-correspondence $\m F_{\vartheta}:=E^* \bigodot_{\vartheta} F$ from $\m B$ to $\m C$ and a unitary $\nu : E\bigodot \m F_{\vartheta}\to  F$ defined by $$\nu(x'\odot(x^*\odot y)):=\vartheta(x'x^*)y~\mbox{for all}~x,x'\in E~\mbox{and}~y\in F$$
such that we obtain $\vartheta=\nu(\bullet\odot id_{\m
F_{\vartheta}})\nu^*$. It is immediate from Proposition \ref{prop1} that the map  from $\m F_{\mf{K}}$ onto $E^*\bigodot {F}_{\ma
K}\subset \m F_{\vartheta}$ defined by $\langle
x,x'\rangle\mf{i}(\sigma)\mapsto x^*\odot \ma K^\sigma (x')$ for all
$x,x'\in E$ and $\sigma\in S$, is a bilinear unitary. Now  we  identify $\m
F_{\mf{K}}\subset\m F_{\vartheta}$ and we have $\mf{i}(\sigma) \in
\m F_{\vartheta}$ for all $\sigma\in S$. Further, we get
$$\nu(x\odot\langle
x',x''\rangle\mf{i}(\sigma))=\nu(x\odot(x'^*\odot \ma K^\sigma
(x'')))=\vartheta(xx'^*)\ma K^\sigma (x'')=\ma K^\sigma(x\langle
x',x''\rangle)$$ for all $x,x',x''\in E$, where the last equality follows from Proposition \ref{prop1}. Since $E$ is full and $\m
B$ is unital, we get $\ma K^\sigma (x)=\nu(x\odot \mf{i}(\sigma))$
for $x\in E$.

For each $\sigma\in S$ we have $\begin{pmatrix}
{\mf{i}(\sigma)} & {}  \\
{} & \nu^*
\end{pmatrix}\in \ma B^r \begin{pmatrix}
\begin{pmatrix}{\m C} \\F \end{pmatrix} ,\begin{pmatrix}{\m B}\\E \end{pmatrix}\bigodot \m F_{\vartheta}
\end{pmatrix}$. Since \\
$  \left(\begin{pmatrix}
{b} & {{x}^*}  \\
{x'} & a
\end{pmatrix}\odot id_{\m F_{\vartheta}}
 \right)\begin{pmatrix}
{\mf{i}(\sigma)} & {}  \\
{} & \nu^*
\end{pmatrix}\begin{pmatrix}{c} \\{y} \end{pmatrix} =\begin{pmatrix}
{b\mf{i}(\sigma) c+(x^*\odot id_{\m F_{\vartheta}})\nu^* y}   \\
{x'\odot \mf{i}(\sigma) c+(a\odot id_{\m F_{\vartheta}})\nu^* y}
\end{pmatrix}$
we have
\begin{eqnarray*}
 &&\left<\left(\begin{pmatrix}
{b_1} & {{x}^*_1}  \\
{x'_1} & a_1
\end{pmatrix}\odot id_{\m F_{\vartheta}}
 \right)\begin{pmatrix}
{\mf{i}(\sigma)} & {}  \\
{} & \nu^*
\end{pmatrix}\begin{pmatrix}{c_1} \\{y_1} \end{pmatrix} ,  \left(\begin{pmatrix}
{b_2} & {{x}^*_2}  \\
{x'_2} & a_2
\end{pmatrix}\odot id_{\m F_{\vartheta}}
 \right)\begin{pmatrix}
{\mf{i}(\sigma')} & {}  \\
{} & \nu^*
\end{pmatrix}\begin{pmatrix}{c_2} \\{y_2} \end{pmatrix}\right>
 \end{eqnarray*}
\begin{eqnarray*}
&=& c^*_1\langle\mf{i}(\sigma),b^*_1b_2\zeta_j\rangle c_2
+c^*_1\langle \mf{i}(\sigma),b^*_1 (x^*_2 \odot id_{\m
F_{\vartheta}})\nu^* y_2\rangle+\langle (x^*_1\odot id_{\m
F_{\vartheta}})\nu^* y_1,b_2\mf{i}(\sigma')\rangle c_2
\\ &&+ \langle (x^*_1\odot id_{\m F_{\vartheta}})\nu^* y_1,(x^*_2\odot id_{\m F_{\vartheta}})\nu^* y_2\rangle
+c^*_1\langle
x'_1\odot\mf{i}(\sigma),x'_2\odot\mf{i}(\sigma')\rangle c_2
\\ &&+ c^*_1\langle x'_1\odot\mf{i}(\sigma),(a_2\odot id_{\m F_{\vartheta}})\nu^* y_2\rangle +\langle(a_1 \odot id_{\m F_{\vartheta}})\nu^* y_1,x'_2\odot \mf{i}(\sigma')\rangle c_2
\\&&+\langle(a_1 \odot id_{\m F_{\vartheta}})\nu^* y_1,(a_2\odot id_{\m F_{\vartheta}})\nu^* y_2\rangle
\\ &=& c^*_1\mf{K}^{\sigma,\sigma'}(b^*_1 b_2)c_2 +c^*_1\langle \ma K^{\sigma} (x_2 b_1),y_2\rangle+\langle y_1, \ma K^{\sigma'}(x_1 b_2)\rangle c_2+\langle y_1, \vartheta(x_1 x^*_2)y_2\rangle
\\ &&+ c^*_1\mf{K}^{\sigma,\sigma'}(\langle x'_1, x'_2\rangle)c_2+c^*_1\langle \ma K^{\sigma} (a^*_2 x'_1),y_2\rangle+\langle y_1, \ma K^{\sigma'}(a^*_1 x'_2)\rangle c_2+\langle y_1, \vartheta(a^*_1 a_2)y_2\rangle
\\ &=& \left<\begin{pmatrix}{c_1} \\{y_1} \end{pmatrix},
 \begin{pmatrix}
{\mf{K}^{\sigma,\sigma'}} &  {\ma K^{\sigma^*}} \\
 \ma K^{\sigma'} & \vartheta
\end{pmatrix} \left(\begin{pmatrix}
{b_1} & {{x}^*_1}  \\
{x'_1} & a_1
\end{pmatrix}^*\begin{pmatrix}
{b_2} & {{x}^*_2}  \\
{x'_2} & a_2
\end{pmatrix} \right)   \begin{pmatrix}{c_2} \\{y_2} \end{pmatrix}\right>
\end{eqnarray*}
for all $x_1,x_2,x'_1,x'_2\in E,~ b_1,b_2\in \m B,~c_1,c_2\in\m
C,~y_1,y_2\in F,~a_1,a_2\in \ma B^a (E)$. Therefore
$\begin{pmatrix}
{\mf{K}^{\sigma,\sigma'}} &  {\ma K^{\sigma^*}} \\
 \ma K^{\sigma'} & \vartheta
\end{pmatrix}$ forms a CPD-kernel and hence $\{\ma K^{\sigma}\}_{\sigma\in S}$ is a strictly CPDH$_0$-family.
\end{proof}

We further generalize the notion of CPDH-dilation as follows:

\begin{definition}
Suppose ${E}$ and ${F}$ are Hilbert $C^*$-modules over $C^*$-algebras $\m B$ and $\m C$ respectively. Let $\mf{K}:S\times S\to \ma B(\m B,\m C)$ be a
  CPD-kernel. Let $\mf{P}$ be a CPD-kernel over the set $E$ from $\ma B^a (E)$ to $\m B$ and let $\mf{L}$ be a CPD-kernel
  over the set $\{\ma K^{\sigma}(x):\sigma\in S,x\in E\}$ from $\ma B^a (F)$ to $\m C$. A homomorphism $\vartheta: \ma B^a (E)\to \ma B^a (F)$ is called a
{\rm generalized CPDH-dilation} of $\mf{K}$ if $E$
  is full and if $\{\ma K^{\sigma}\}_{\sigma\in S}$ is a collection of linear maps from ${E}$ to ${F}$ such that the following diagram commutes for all
  $x,x'\in E$ and $\sigma,\sigma'\in S$:
\begin{eqnarray}\label{diag3}
\xymatrix{ \m B \ar@{->}[r]^{ \mf{K}^{\sigma,\sigma'}}
     \ar@{<-}[d]_{{\mf P}^{x, x'}}
&\m C
     \ar@{<-}[d]^{\mf{L}^{\ma K^{\sigma}(x),\ma K^{\sigma'}(x')}}
\\
\ma B^a (E) \ar@{->}[r]^{\vartheta}
     & \ma B^a (F)
}
\end{eqnarray}
The generalized CPDH-dilation $\theta$ is called {\rm quasi-dilation} if $E$ is not
necessarily full.
\end{definition}

Let $\mf{L}$ be  a CPD-kernel over the set $S'=\{\ma K^{\sigma}(x):\sigma\in S,x\in E\}$ from a unital $C^*$-algebra $\ma
B^a (F)$ to a $C^*$-algebra $\m C$. We get the Kolmogorov
decomposition $(\m F, \mf{i})$ such that
$$\langle \mf{i}(y), a \mf{i}(y')\rangle=\mf{L}^{y,y'}(a)~\mbox{for all $y,y'\in S'$, $a\in \ma B^a (F)$}~$$
and  $\m F=[\{a \mf{i}(y) c:a\in \ma B^a (F),~y\in S',~c\in
C\}]$. Hence we get $$\mf{K}^{\sigma,\sigma'}({\mf P}^{x,
x'}(a))=\langle \mf{i}(\ma K^{\sigma}(x)),\vartheta(a)\mf{i}(\ma
K^{\sigma'}(x'))\rangle$$ for each $\sigma,\sigma'\in S$, $x,x'\in
E$ and $a\in \ma B^a (F)$.
We denote the homomorphism which gives the left action on $\m F$ by
$\theta:\ma B^a (F)\to \ma B^a (\m F)$. Observe that the following diagram
commutes for all $x,x'\in E$ and  $\sigma,\sigma'\in S$:
\begin{eqnarray*}\label{diag4}
\xymatrix{ \m B \ar@{->}[r]^{ \mf{K}^{\sigma,\sigma'}}
     \ar@{<-}[d]_{{\mf P}^{x, x'}}
&\m C
     \ar@{<-}[d]^{\langle \mf{i}(\ma K^{\sigma} (x)),\bullet \mf{i}(\ma K^{\sigma'}(x'))\rangle}
\\
\ma B^a (E) \ar@{->}[r]^{\theta\circ \vartheta}
     & \ma B^a (\m F)
}
\end{eqnarray*}

\begin{proposition}
Suppose ${E}$ and ${F}$ are Hilbert $C^*$-modules over $C^*$-algebras $\m B$ and $\m C$ respectively. Let $\mf{K}:S\times S\to \ma B(\m B,\m C)$ be a CPD-kernel. 
Let $\mf{P}$ be a CPD-kernel over the set $E$ from $\ma B^a (E)$ to $\m B$ defined by $\mf{P}^{x,x'}:=\langle x,\bullet x'\rangle$
 where $x,x'\in E$ and let $\mf{L}$ be a CPD-kernel over the set $\{\ma K^{\sigma}(x):\sigma\in S,x\in E\}$ from $\ma B^a (F)$ to $\m C$. 
If $\vartheta: \ma B^a (E)\to \ma B^a (F)$ is a generalized
 quasi-CPDH-dilation of $\mf{K}$ with respect to CPD-kernels $\mf{P}$ and $\mf{L}$, then
 $\theta\circ \vartheta:\ma B^a (E)\to \ma B^a (\m F)$ is a quasi-CPDH-dilation of $\mf{K}$ with respect to maps
 $\{\mf{i}\circ\ma K^{\sigma} :E\to \m F\}_{\sigma\in S}$ where $(\m F,\mf{i})$ is the Kolmogorov decomposition of $\mf L$ and
$\theta:\ma B^a (F)\to \ma B^a (\m F)$ is a homomorphism which gives
 the left action on $\m F$.
\end{proposition}

Let $\m B$ be a $C^*$-algebra. Given two CPD-kernels $\mf{K}$ and $\mf{L}$ over a set $S$ on  $\m B$, we define the {\it Schur product} as the kernel 
$\mf{K}\circ\mf{L}$ over $S$ on 
$\m B$ by $(\mf{K}\circ\mf{L})^{\sigma,\sigma'}:=\mf{K}^{\sigma,\sigma'}\circ\mf{L}^{\sigma,\sigma'}$ for $\sigma,\sigma'\in S$. Using the Kolmogorov decomposition 
it is clear that the kernel $\mf{K}\circ\mf{L}$ over $S$ on $\m B$ is a CPD-kernel. Let us denote by $\mathbb{T}$ the semigroup $\mathbb{N}_0$ or $\mathbb{R}_{+}$.  
A collection of CPD-kernels $\{\mf{K}_t\}_{t\in \mathbb T}$ over $S$ on $\m B$ forms a {\it semigroup of CPD-kernels} or {\it CPD-semigroup}
if $\mf{K}_t\circ\mf{K}_s:=\mf{K}_{t+s}$ for $t,s\in \mathbb{T}$. 
The semigroup  is denoted by $\mf{K}=(\mf{K}_t)_{t\in\mathbb{T}}$. We define similar notion of dilation for semigroups of CPD-kernels, as given above for CPD-kernels. 
The theory of CP-semigroups finds significant applications in quantum statistical mechanics, quantum probability theory, etc. and 
many of the aspects of this theory can be extended to CPD-semigroups.

\begin{definition}
Let $E$ be a Hilbert $C^*$-module on a $C^*$-algebra $\m B$, $S$ be a set and $\ma K^{\sigma}_t$  be a map from $E$ to $E$
for each $\sigma\in S$ and $t\in \mathbb{T}$.   A semigroup $\{\{\ma
{K}^{\sigma}_t\}_{\sigma\in S}:t\in \mathbb{T}\}$ is called
  \begin{enumerate}[(a)]
   \item a {\rm CPD-semigroup} on $E$ if it extends to a semigroup of CPD-kernels  $\begin{pmatrix}
{\mf{K}^{\sigma,\sigma'}_t} &  {\ma K^{\sigma^*}_t} \\
 \ma K^{\sigma'}_t & \vartheta_t
\end{pmatrix}$ acting block-wise on the linking algebra of $E$.
\item a {\rm CPDH-semigroup} on $E$ if it is a CPD-semigroup where $\vartheta_t$  can be chosen to form an $E$-semigroup and 
for each $t$, the  kernel $\{\mf{K}^{\sigma,\sigma'}_t: \sigma,\sigma \in S  \}$
 can be chosen such that $\{\ma K^\sigma_t\}_{\sigma\in S}$ is a $\mf K_t$-family.
  \end{enumerate}

\end{definition}

\begin{definition}
  Let ${E}$ be a Hilbert $C^*$-module over a $C^*$-algebra $\m B$. Let $S$ be a set and let $\mf K$ be a CPD-semigroup over a set $S$ on  $\m B$. 
A semigroup of $*$-homomorphisms $\vartheta_t: \ma B^a (E)\to \ma B^a (F)$ for $t\in \mathbb{T}$ is called  a {\rm CPDH-dilation} of $\mf K$ if $E$ is full and if there exists CPDH-semigroup on $E$ consisting of linear maps ${\ma K}^\sigma_t: {E} \to{E}$ for each $\sigma\in S$ such that the diagram

\begin{eqnarray}\label{diag5}
\xymatrix{ \m B \ar@{->}[r]^{ \mf{K}^{\sigma,\sigma'}_t}
     \ar@{<-}[d]_{\langle x,\bullet x'\rangle}
&\m C
     \ar@{<-}[d]^{\langle \ma K^{\sigma}_t (x),\bullet \ma K^{\sigma'}_t(x')\rangle}
\\
\ma B^a (E) \ar@{->}[r]^{\vartheta_t}
     & \ma B^a (F)
}
\end{eqnarray}
commutes for all $x,x'\in E$. We say that the CPDH-dilation is {\rm quasi-dilation} if $E$ is not
necessarily full. Further if each $\vartheta_t$ is strict, then we
call a dilation {\rm strict}. We say that the CPDH-(quasi-) dilation
is {\rm CPDH$_0$-(quasi-) dilation} if each $\vartheta_t$ is unital.
\end{definition}

Now we construct a CPDH-dilation for a given CPD-semigroup using the concept of product systems: 
Let $\m B$ be a $C^*$-algebra and for each $t\in \mathbb{T}$, $E_t$ be a $C^*$-correspondence from $\m B$ to $\m B$ with $E_0=\m B$. The family 
$E^{\odot}=(E_t)_{t\in \mathbb{T}}$ is called a {\it product system of $C^*$-correspondences} if there exists an associative product $$(x_s,y_t)
\mapsto x_sy_t:=u_{s,t}(x_s\odot y_t)\in E_{s+t}~\mbox{for}~x_s\in E_s, y_t\in E_t,~ s,t\in \mathbb{T};$$ where $u_{s,t}:E_s\bigodot
E_t\to E_{s+t}$ are bilinear unitaries, and for each $t\in \mathbb{T}$ maps $u_{0,t}$ and $u_{t,0}$ are left and right actions, respectively. A {\it unit} of 
the product system $E^{\odot}$ is a family $\xi^{\sigma\odot}=(\xi^{\sigma}_t)_{t\in \mathbb{T}}$ for each $\sigma\in S$ satisfying 
$\xi^{\sigma}_s\xi^{\sigma}_t=\xi^{\sigma}_{s+t}$ for $s,t\in \mathbb{T}$. Let $E$ be a full Hilbert $\m B$-module. A {\it} left dilation of $E^{\odot}$ 
to $E$ is a family of unitaries $\nu_t:E\bigodot E_t\to E$ which satisfy the associativity condition $\nu_t(\nu_s(x \odot y_s) \odot z_t)=
\nu_{s+t}(x \odot u_{s,t}(y_s \odot z_t))$.
Let $\mf{K}$ be a CPD-semigroup over a set $S$ on a unital $C^*$-algebra $\m B$. In Section 4.3 of
\cite{BBLS04} it is shown that there exists a product system $E^{\odot}=(E_t)_{t\in \mathbb{T}}$ of
$\m B$-correspondences and there exists a unit $\xi^{\sigma\odot}$ in $E^{\odot}$ such that $\langle
\xi^{\sigma}_t,b \xi^{\sigma'}_t\rangle=\mf{K}^{\sigma,\sigma'}_t(b)$
for all $b\in \m B,~\sigma,\sigma'\in S,~t\in\mathbb{T}$. Let $E$ denote the inductive limit over $E_t$.
For every fix $t\in \mathbb{T}$, the unitary $\nu_t:E\bigodot E_t\to E$ is obtained as a limit of unitaries $u_{s,t}:E_s\bigodot
E_t\to E_{s+t}$ as $s\to\infty$. These unitaries $\nu_t$
form a left dilation of $E^{\odot}$ to
$E$ (cf. Theorem 2.2 of \cite{Sk07}). Therefore  $\vartheta$ on $\ma B^a
(E)$ defined as $\vartheta_t (a)=\nu_t (a\odot id_{t})\nu_t$ for all
$a\in \ma B^a(E), t\in \mathbb{T},$ is  an $E_0$-semigroup. Note that if we set $\ma K^\sigma_t
(x):=\nu_t (x\odot\xi^\sigma_t) $ for all $\sigma\in S$, then the
Diagram \ref{diag5} commutes.

\vspace{0.5cm} {\bf Acknowledgements:} The second author would like to thank K. Sumesh for several discussions.

\bibliographystyle{amsalpha}{
\bibliography{harshbib}}

\providecommand{\bysame}{\leavevmode\hbox to3em{\hrulefill}\thinspace}
\providecommand{\MR}{\relax\ifhmode\unskip\space\fi MR }
\providecommand{\MRhref}[2]{%
  \href{http://www.ams.org/mathscinet-getitem?mr=#1}{#2}
}
\providecommand{\href}[2]{#2}
\begin{thebibliography}{EKQR00}

\bibitem[AK01]{AK01}
L.~Accardi and S.~V. Kozyrev, \emph{On the structure of {M}arkov flows}, Chaos
  Solitons Fractals \textbf{12} (2001), no.~14-15, 2639--2655, Irreversibility,
  probability and complexity (Les Treilles/Clausthal, 1999). \MR{1857648
  (2002h:46110)}

\bibitem[BBLS04]{BBLS04}
Stephen~D. Barreto, B.~V.~Rajarama Bhat, Volkmar Liebscher, and Michael Skeide,
  \emph{Type {I} product systems of {H}ilbert modules}, J. Funct. Anal.
  \textbf{212} (2004), no.~1, 121--181. \MR{2065240 (2005d:46147)}

\bibitem[BRS12]{BRS12}
B.~V.~Rajarama Bhat, G.~Ramesh, and K.~Sumesh, \emph{Stinespring's theorem for
  maps on {H}ilbert {$C^\ast$}-modules}, J. Operator Theory \textbf{68} (2012),
  no.~1, 173--178. \MR{2966040}

\bibitem[EKQR00]{EKQR00}
Siegfried Echterhoff, S.~Kaliszewski, John Quigg, and Iain Raeburn,
  \emph{Naturality and induced representations}, Bull. Austral. Math. Soc.
  \textbf{61} (2000), no.~3, 415--438. \MR{1762638 (2001j:46101)}

\bibitem[Heo99]{He99}
Jaeseong Heo, \emph{Completely multi-positive linear maps and representations
  on {H}ilbert {$C^*$}-modules}, J. Operator Theory \textbf{41} (1999), no.~1,
  3--22. \MR{1675235 (2000a:46103)}

\bibitem[Joi11]{J11}
Maria Joi{\c{t}}a, \emph{Covariant version of the {S}tinespring type theorem
  for {H}ilbert {$C^*$}-modules}, Cent. Eur. J. Math. \textbf{9} (2011), no.~4,
  803--813. \MR{2805314 (2012f:46110)}

\bibitem[Kas88]{Kas88}
G.~G. Kasparov, \emph{Equivariant {$KK$}-theory and the {N}ovikov conjecture},
  Invent. Math. \textbf{91} (1988), no.~1, 147--201. \MR{918241 (88j:58123)}

\bibitem[Lan95]{La95}
E.~C. Lance, \emph{Hilbert {$C^*$}-modules}, London Mathematical Society
  Lecture Note Series, vol. 210, Cambridge University Press, Cambridge, 1995, A
  toolkit for operator algebraists. \MR{1325694 (96k:46100)}

\bibitem[MSS06]{MSS06}
Paul~S. Muhly, Michael Skeide, and Baruch Solel, \emph{Representations of
  {${\scr B}^a(E)$}}, Infin. Dimens. Anal. Quantum Probab. Relat. Top.
  \textbf{9} (2006), no.~1, 47--66. \MR{2214501 (2006m:46074)}

\bibitem[Pas73]{Pas73}
William~L. Paschke, \emph{Inner product modules over {$B^{\ast} $}-algebras},
  Trans. Amer. Math. Soc. \textbf{182} (1973), 443--468. \MR{0355613 (50
  \#8087)}

\bibitem[Rie74]{Ri74}
Marc~A. Rieffel, \emph{Induced representations of {$C^{\ast} $}-algebras},
  Advances in Math. \textbf{13} (1974), 176--257. \MR{0353003 (50 \#5489)}

\bibitem[Ske01]{Sk01}
Michael Skeide, \emph{Hilbert modules and applications in quantum probability},
  Habilitationsschrift (2001).

\bibitem[Ske07]{Sk07}
\bysame, \emph{{$E_0$}-semigroups for continuous product systems}, Infin.
  Dimens. Anal. Quantum Probab. Relat. Top. \textbf{10} (2007), no.~3,
  381--395. \MR{2354367 (2009b:46138)}

\bibitem[Ske11]{Sk11}
\bysame, \emph{Hilbert modules---square roots of positive maps}, Quantum
  probability and related topics, QP--PQ: Quantum Probab. White Noise Anal.,
  vol.~27, World Sci. Publ., Hackensack, NJ, 2011, pp.~296--322. \MR{2799131}

\bibitem[Ske12]{Sk12}
\bysame, \emph{A factorization theorem for {$\phi$}-maps}, J. Operator Theory
  \textbf{68} (2012), no.~2, 543--547. \MR{2995734}

\bibitem[SS14]{SSu14}
Michael Skeide and K.~Sumesh, \emph{C{P}-{H}-extendable maps between {H}ilbert
  modules and {CPH}-semigroups}, J. Math. Anal. Appl. \textbf{414} (2014),
  no.~2, 886--913. \MR{3168002}

\bibitem[Wil07]{Wi07}
Dana~P. Williams, \emph{Crossed products of {$C{^\ast}$}-algebras},
  Mathematical Surveys and Monographs, vol. 134, American Mathematical Society,
  Providence, RI, 2007. \MR{2288954 (2007m:46003)}

\end{thebibliography}

\vspace{1cm}
\noindent Department of Mathematics, Indian Institute of Technology Bombay, Powai, Mumbai-400076, India\\
Email address: dey@math.iitb.ac.in

\vspace{.5cm}

\noindent Department of Mathematics, Indian Institute of Technology Bombay, Powai, Mumbai-400076, India\\
Email address: harsh@math.iitb.ac.in

\end{document}